\documentclass[letterpaper,11pt,reqno]{amsart} 
\usepackage[portrait,margin=1in]{geometry} 
\usepackage{mathrsfs,xfrac} 
\usepackage[colorlinks=true,linkcolor=blue,citecolor=blue,urlcolor=blue]{hyperref} 
\usepackage{amsmath,amssymb,amsthm,amsfonts,amsbsy,latexsym,dsfont,color} 
\usepackage[numeric,initials,nobysame]{amsrefs} 
\usepackage[foot]{amsaddr}
\usepackage{enumitem}

\newtheorem{thm}{Theorem}[section] 
\newtheorem{lem}[thm]{Lemma} 
\newtheorem{cor}[thm]{Corollary} 
\newtheorem{prop}[thm]{Proposition} 
\newtheorem{defn}[thm]{Definition} 
\newtheorem{rem}[thm]{Remark} 
 
\newtheorem{ass}[thm]{Assumption} 
\newtheorem{ex}[thm]{Example} 
\newtheorem{conj}[thm]{Conjecture} 

\renewcommand{\le}{\leqslant} 
\renewcommand{\ge}{\geqslant} 
\newcommand{\T}{\intercal} 
\newcommand{\half}{\ensuremath{\sfrac12}} 
 
\newcommand{\ra}{\rangle} 
\newcommand{\la}{\langle}

\newcommand{\ind}{\mathds{1}} 
\newcommand{\eps}{\varepsilon} 
 
\newcommand{\norm}[1]{\Vert#1\Vert} 
\newcommand{\abs}[1]{\left\vert#1\right\vert}

\newcommand{\ie}{\emph{i.e.,}} 
 
\newcommand{\eg}{\emph{e.g.,}}

\def\qed{ \hfill $\blacksquare$} 

\let\ga=\alpha \let\gb=\beta \let\gc=\gamma \let\gd=\delta   \let\gh=\eta  \let\gk=\kappa \let\gl=\lambda  \let\gn=\nu   \let\gr=\rho \let\gs=\sigma \let\gt=\tau    \let\gz=\zeta \let\gC=\Gamma \let\gD=\Delta  \let\gL=\Lambda     \let\gS=\Sigma    

\newcommand{\cC}{\mathcal{C}}

\newcommand{\cP}{\mathcal{P}} 
\newcommand{\cQ}{\mathcal{Q}}

\newcommand{\cV}{\mathcal{V}}

\newcommand{\vone}{\mathbf{1}}

\newcommand{\vQ}{\mathbf{Q}} 
\newcommand{\vR}{\mathbf{R}}

\newcommand{\vX}{\mathbf{X}}

\newcommand{\ve}{\mathbf{e}}

\newcommand{\vp}{\mathbf{p}} 
\newcommand{\vq}{\mathbf{q}}

\newcommand{\vu}{\mathbf{u}} 
\newcommand{\vv}{\mathbf{v}} 
 
\newcommand{\vx}{\mathbf{x}} 
\newcommand{\vy}{\mathbf{y}}

\newcommand{\mvu}{\boldsymbol{u}} 
\newcommand{\mvv}{\boldsymbol{v}} 
 
\newcommand{\mvx}{\boldsymbol{x}}

\newcommand{\mvgr}{\boldsymbol{\rho}} 
\newcommand{\mvgs}{\boldsymbol{\sigma}} 
 
\newcommand{\mvgt}{\boldsymbol{\tau}}

\newcommand{\mvgz}{\boldsymbol{\zeta}}

\newcommand{\mveta}{\boldsymbol{\eta}}

\newcommand{\mvrho}{\boldsymbol{\rho}}

\newcommand{\mvxi}{\boldsymbol{\xi}}

\newcommand{\fR}{\mathfrak{R}} 
\newcommand{\mfR}{\boldsymbol{\mathfrak{R}}}

\newcommand{\bN}{\mathbb{N}}

\newcommand{\bR}{\mathbb{R}}

\newcommand{\dR}{\mathds{R}}

\newcommand{\sP}{\mathscr{P}}

\DeclareMathOperator{\E}{\mathds{E}} 
\DeclareMathOperator{\pr}{\mathds{P}} 
 
\DeclareMathOperator{\var}{Var}

\DeclareMathOperator{\tr}{Tr} 
\DeclareMathOperator{\sech}{sech} 
\DeclareMathOperator{\diag}{diag} 
\DeclareMathOperator{\N}{N} 
\DeclareMathOperator{\RS}{RS}
\DeclareMathOperator{\sym}{sym}
\DeclareMathOperator{\vect}{vec}
\newcommand{\ovR}{\overline{\vR}} 
\newcommand{\oR}{\overline{R}} 
\newcommand{\ov}{\overline} 
\newcommand{\indf}{\emph{indefinite}}
\newcommand{\pd}{\emph{positive-definite}}

\begin{document}

\title[Replica Symmetric MSK Model]{Fluctuation Results for Multi-species Sherrington-Kirkpatrick model in the Replica Symmetric Regime}
\author[Dey]{Partha S.~Dey$^\star$}
\author[Wu]{Qiang Wu$^\dagger$}
\address{Department of Mathematics, University of Illinois at Urbana-Champaign, 1409 W Green Street, Urbana, Illinois 61801}
\email{$^\star$psdey@illinois.edu, $^\dagger$qiangwu2@illinois.edu}
\date{\today}

\begin{abstract}
	We study the Replica Symmetric region of general multi-species Sherrington-Kirkpatrick (MSK) Model and answer some of the questions raised in Ann.~Probab.~43~(2015), no.~6, 3494--3513, where the author proved the Parisi formula under \emph{positive-definite} assumption on the disorder covariance matrix $\Delta^2$.
	First, we prove exponential overlap concentration at high temperature for both \indf~and \emph{positive-definite} $\Delta^2$ MSK model. We also prove a central limit theorem for the free energy using overlap concentration. Furthermore, in the zero external field case, we use a quadratic coupling argument to prove overlap concentration up to $\beta_c$, which is expected to be the critical inverse temperature. The argument holds for both \emph{positive-definite} and emph{indefinite} $\Delta^2$, and $\beta_c$ has the same expression in two different cases.
	Second, we develop a species-wise cavity approach to study the overlap fluctuation, and the asymptotic variance-covariance matrix of overlap is obtained as the solution to a matrix-valued linear system. The asymptotic variance also suggests the de Almeida--Thouless (AT) line condition from the Replica Symmetry (RS) side. Our species-wise cavity approach does not require the positive-definiteness of $\Delta^2$. However, it seems that the AT line conditions in \emph{positive-definite} and \emph{indefinite} cases are different.
	Finally, in the case of \emph{positive-definite} $\Delta^2$, we prove that above the AT line, the MSK model is in Replica Symmetry Breaking phase under some natural assumption. This generalizes the results of J.~Stat.~Phys.~174 (2019), no.~2, 333--350, from 2-species to general species.
\end{abstract}

\maketitle

\section{Introduction and Main Results}\label{sec:intro}


Multi-species Sherrington-Kirkpatrick (MSK) model, as a generalization of the classical SK model~\cites{Tal11a,Tal11b,Pan13}, was introduced by Barra~\emph{et.~al.}~in~\cite{BCMT15}. In the MSK model, the spins are divided into finitely many species, and the density of each species is asymptotically fixed with increasing system volume. The interaction parameter between different spins depends only on the species structure. Thus, one can consider the MSK model as an inhomogeneous generalization of the classical SK model. The detailed definition is given in Section~\ref{s1sec:msk}.

One important question in spin glasses is to understand the behavior of limiting free energy. In the classical SK model, Parisi~\cites{Par83, Par80} proposed his celebrated variational formula, known as the Parisi formula, which gives the limiting free energy at all temperatures. Talagrand~\cite{Tal06} proved this in the setting of mixed $p$-spin models for even $p$. Later Panchenko~\cite{Pan14}, using a different approach, proved the Parisi formula of mixed $p$-spin models for all $p$.

The authors in~\cite{BCMT15} proposed a Parisi type formula for the limiting free energy in the MSK model. They also proved the upper bound by adapting Guerra's interpolation method under \pd~condition on the interaction matrix $\gD^2$. Later, Panchenko~\cite{Pan15} completed the proof by proving a multi-species version of the Ghirlanda-Guerra identities to match the lower bound. The \pd~condition on $\gD^{2}$ was used only in the proof of the upper bound. In~\cite{Pan15}, Panchenko also proposed several open questions, such as understanding behaviors of the model when the interaction matrix $\gD^2$ is \emph{indefinite}, deriving an analog of the AT line condition in MSK model, among others. The authors in~\cite{BSS19} derived the AT line condition of two-species \pd~case and proved that above the AT line, the 2-species SK model is RSB. The essential tool of their argument is a perturbation of the Parisi formula, which is only known for \pd~$\gD^2$. The idea fails in the general species case because of algebraic difficulties.

For the interesting question about what happens when $\gD^2$ is \indf, to the best of our knowledge, only some partial results on a few particular models have been known very recently. In~\cites{BGG11, BGGPT14}, there are some conjectured min-max type formulas by physicists for the bipartite SK model, which is a special case of \indf~MSK model. Another particular case of \indf~MSK model known as deep Boltzmann machine (DBM) was investigated in \cites{ABCM20, ACCM21} in the replica symmetry regime, and a complete solution was obtained in \cite{ACCM20} under assumptions on the Nishimori line. Besides that, a min-max formula for the replica symmetric solution in the DBM model is proved in~\cite{Gen20}. Recently, Mourrat~\cites{MD20, Mou20A, Mou20B, Mou20C, Mou18, Mou19} has reinterpreted the Parisi formula as the solution to an infinite-dimensional Hamilton-Jacobi equation in the Wasserstein space of probability measures on the positive half-line. Particularly in~\cite{Mou20B}, he studied the bipartite SK model and questioned the possibility to represent the limiting free energy as a variational formula. In the setting of the Hamilton-Jacobi equation, non-convex $\gD^2$ breaks down the application of the Hopf-Lax formula while solving the PDE. For the spherical case, Auffinger and Chen~\cite{AC14} studied the bipartite spherical SK (BSSK) model and proved a variational formula for the limiting free energy at sufficiently high temperature. Later, Baik and Lee~\cite{BL17} computed the limiting free energy and its fluctuations at all non-critical temperatures in the BSSK model using tools from random matrix theory. However, computing limiting free energy for general \indf~model is still unknown. Specifically for Ising spins, we do not even know the limiting free energy at any temperature, other than the lower bound by Panchenko~\cite{Pan15}.

In this paper, we do a high-temperature analysis of the Ising MSK model, and our main contributions are:
\begin{enumerate}
	\item For external field $h \ge 0$, by extending Latala's argument, we prove an RS regime in the MSK model, where the overlap has exponential concentration. Note that our approach is unified for both \pd~and \indf~MSK model, but the proved RS regimes have a different form; see Theorem~\ref{thm1}. For $h=0$, by using a quadratic coupling argument, we prove the concentration of overlap up to $\gb_c=\rho(\gD^2\gL)^{-\half}$, where $\gD^{2}$ is the interaction matrix and $\gL$ is the diagonal matrix containing the species size ratios. This result is true for \indf~and \pd~$\gD^2$, \ie~we prove the whole RS regime for the general MSK model, which is given as $\gb < \gb_c$, see Theorem~\ref{ovp no h}. The above overlap concentrations also enable us to prove the limit and central limit theorem of free energy in the corresponding RS regime; see Theorem~\ref{RS solution}.

	\item By developing a different species-wise cavity method, we derive a linear system of Gibbs average of a quadratic form of overlap vectors. The system is solved using linear algebraic methods, enabling us to compute the variance-covariance structure of overlap vectors. The computation also suggests the AT-line condition in the MSK model from the replica symmetry side. Note that our species-wise cavity method does not require \pd~$\gD^2$, but the AT line condition for \indf~$\gD^2$ seems to be more complicated; see the discussions below Theorem~\ref{thm:varoverlap}.

	\item In the case of \pd~$\gD^2$, we prove the AT line condition from the RSB side under some natural assumption. The key is still the perturbation idea of the Parisi formula. This generalizes the result for $2$-species SK model in~\cite{BSS19}. For the \indf~$\gD^2$ case, we use our species-wise cavity approach to give a conjectured form of the AT line condition. To get a rigorous proof of the AT line in this case is challenging. Because, first, the Parisi formula for \indf~MSK model is still in mystery, so the classical perturbation technique can not be used. Second, proving the uniqueness of stationary point is hard. We prove a uniqueness result in 2-species case using an elementary approach, see Proposition~\ref{prop:indf-uniq}.
\end{enumerate}

The definition of the MSK model is given in Section~\ref{s1sec:msk}, then we review the main results of the $2$-species case in~\cite{BSS19}, and conclude Section~\ref{sec:intro} by the statement of our results. Readers can always check notations back in Section~\ref{notation}. A road map is given in Section~\ref{roadmap}.

\subsection{MSK model}\label{s1sec:msk}
Fix $N \ge 1$. For a spin configuration on the $N$-dimensional hyper-cube, $\mvgs =( \gs_1, \gs_2, \ldots, \gs_N) \in \gS_N := \{-1,+1\}^N$, we consider the Hamiltonian given by
\begin{align}\label{hamilton}
	H_N(\mvgs) := \frac{\gb}{\sqrt{N}}\sum_{1\le i<j\le N} g_{ij} \gs_i \gs_j+ h \sum_{i=1}^N \gs_i
\end{align}
where $g_{ij}$, the disorder interaction parameters are independent centered Gaussian random variables, $\gb>0$ is the inverse temperature and $h \ge 0$ is the external field.

In the classical SK model, the variance structure of disorder $g_{ij}$ is homogeneous, usually taken as $g_{ij} \sim \N(0,1)$. While in the MSK model, the variance of $g_{ij}$ depends on the structure of species among $N$ spins.

Assume that, there are $m  \ge  2$ species. When $m=1$ the MSK model reduces to the classical SK model. We partition the set of spins into $m$ disjoint sets, namely,
\[
	I = \bigcup_{s=1}^m I_s = \{1,2,\ldots, N\},\qquad I_{s}\cap I_{t}=\emptyset\text{ for } s\neq t.
\]
For $i \in I_s, j \in I_t$, we assume
\[ \E g_{ij}^2 = \gD_{st}^2 ,\]
\ie~the inhomogeneity of MSK model comes from the interaction among different species. While proving MSK Parisi formula in~\cites{Pan15,BCMT15}, the assumption that $\gD^2 = (\!(\gD_{st}^2)\!)_{s,t=1}^m$ is symmetric and \pd\ was used. But in this paper, we do not require the positive-definiteness condition for most of the results. Besides that, we assume that the ratio of spins in each species is fixed asymptotically, \ie~for $s=1,2,\ldots,m$ and
\[ \gl_{s,N}:={|I_s|}/{N},\]
we have
\[ \lim_{N\to \infty} \gl_{s,N} = \gl_s \in (0,1). \]
Since $\gl_{s,N}$ and $\gl_s$ are asymptotically the same, for the rest of the article we will use $\gl_s$ instead of $\gl_{s,N}$ for convenience. We denote $\gL:=\diag(\gl_1,\gl_2,\ldots, \gl_m)$.

The overlap vector between two replicas $\mvgs^1, \mvgs^2 \in \gS_N$ is given by
\[ \vR_{12} = (R_{12}^{(1)},R_{12}^{(2)},\ldots,R_{12}^{(m)})^{\T} ,\]
where $$R_{12}^{(s)} := \frac{1}{|I_s|} \sum_{i \in I_s} \gs_i^1 \gs_i^2.$$ is the overlap restricted to species $s \in \{1, \ldots, m\}$. In some cases we write it as $R^{(s)}$ for short if there are only 2 replicas involved. All vectors will be considered as a column vector in the rest of the article.

A central question in spin glasses is to understand the free energy
\begin{align}
	F_N := \frac{1}{N} \log Z_N, \quad \text{where} \quad Z_N := \sum_{\mvgs \in \gS_N} \exp(H_N(\mvgs))
\end{align}
is the partition function, and the Gibbs measure is given by
\begin{align}
	G_{N}(\mvgs) = {\exp(H_N(\mvgs))}/{Z_N},\qquad \mvgs\in \gS_{N}.
\end{align}
Later we will also use $\la \cdot \ra$ to denote the Gibbs average just for convenience; check Section~\ref{notation}. When $\gD^{2}$ is \pd, Parisi formula gives a variational representation of the limiting free energy for all $\gb >0, h\in \dR$, which was set up in~\cites{BCMT15,Pan15} for MSK model. This paper is mainly about the phase transitions and fluctuation of free energy in the RS regime. Before stating the main results, let us recall some related results in MSK model in~\cites{BSS19,Pan15,BCMT15}. We will only use the Parisi formula in the proof of Theorem~\ref{AT line}.


\subsection{Related results in MSK model} \label{s1sec:2sk}

Consider a sequence of real numbers
\begin{align}\label{seq1}
	0=\gz_0 < \gz_1 < \cdots < \gz_k < \gz_{k+1}=1,
\end{align}
and for each $s \in \{1,2, \ldots, m\}$, the sequences
\begin{align}\label{seq2}
	0=q_0^s\le q_1^s \le \cdots \le q_{k+1}^s \le q_{k+2}^s = 1.
\end{align}
For $0 \le l \le k+2$, define $\vq_\ell = (q_\ell^s)_{s=1}^{m}$,
\begin{align}\label{Q sequen}
	Q_\ell := \frac 1 2 {\vq_\ell}^{\T} \gL \gD^2 \gL \vq_\ell, \quad Q_\ell^s :=(\gD^2 \gL \vq_\ell)_s \quad \text{for} \ 1 \le s\le m.
\end{align}
Given these sequences, consider i.i.d.~standard Gaussian random variables $(\eta_\ell)_{1\le \ell \le k+2}$. Define
\[ X_{k+2}^s := \log \cosh\bigl(h+\gb \sum_{0 \le \ell \le k+1} \eta_{\ell+1}(Q_{\ell+1}^s-Q_{\ell}^s)^{\half}\bigr) ,\]
then recursively define for $0 \le \ell \le k+1$,
\[ X_\ell^s := \frac{1}{\gz_\ell} \log \E_{\ell+1} \exp(\gz_\ell X_{\ell+1}^s) .\]
where $\E_{\ell+1}$ denotes expectation w.r.t.~$\eta_{\ell+1}$. The following theorem gives the Parisi formula in MSK model.
\begin{thm}
	[{\cite{Pan15}*{Theorem 1}}] For the MSK model with positive-definite $\gD^2$, the limiting free energy is given by
	\begin{align}\label{par}
		\lim_{N \to \infty} F_N = \inf_{\mvgz,\vq}\mathscr{P}(\mvgz,\vq),
	\end{align}
	where
	\[ \sP(\mvgz, \vq) = \log 2+\sum_{s=1}^{m}\gl_s X_0^s - \frac{\gb^2}{2} \sum_{\ell=1}^{k+1} \gz_{\ell}(Q_{\ell+1}-Q_\ell) .\]
	is the Parisi functional  and the $\inf$ in~\eqref{par} is taken over all the sequences in~\eqref{seq1} and~\eqref{seq2}.
\end{thm}
In the above variational formula~\eqref{par}, let $k=0$ and $q_1^s = q_s\in [0,1]$, then the Parisi functional on the RS regime simplifies to
\begin{align*}
	\sP_{\RS}(\vq)
	 & := \log 2 + \sum_{s=1}^m \gl_s \E\log \cosh(\gb \eta\sqrt{(\gD^{2}\gL\vq)_s}+h)+\frac{\gb^2}{4} (\vone-\vq)^{\T}\gL\gD^{2}\gL(\vone-\vq)
\end{align*}
where $\vone$ is the vector of all $1$'s and $\vq=(q_{1},q_{2},\ldots,q_{m})^{\T}$. Taking derivative w.r.t.~$q_t$ for $t=1, \ldots, m$

\begin{align}\label{sys-crit}
	\frac{\partial \sP_{\RS}}{\partial q_t} = \frac{\gb^2}{2}\gl_t \sum_{s=1}^m \gD_{s,t}^2\gl_s\left[q_s - \E \tanh^2(\gb \eta \sqrt{(\gD^{2}\gL\vq)_s} +h)\right] .
\end{align}
and setting it 0, we get the set of critical points (when $\gD^2$ is invertible) to be
\begin{align}\label{crit set}
	\cC(\gb,h) = \left\{ \vq \in [0,1]^m \ \bigl|\ \E \tanh^2 (\gb\eta \sqrt{(\gD^{2}\gL\vq)_s}+h) = q_s \ \text{for} \ s= 1,2, \ldots m\right\}.
\end{align}
For the MSK model with positive-definite $\gD^2$, we define the Replica Symmetric solution as
\begin{align}\label{RS}
	\RS(\gb,h):=\inf_{\vq \in [0,1]^m} \sP_{\RS}(\vq).
\end{align}
Note that $\sP_{\RS}$ is well-defined for \indf~$\gD^2$. However, one expects the replica symmetric solution to be achieved at a saddle point of  $\sP_{\RS}$ instead of a minimizer, see~\cite{BGG11}*{Theorem~4}. The $\sP_{\RS}$ functional in the \indf~case seems to be not convex. We only use the $\RS(\gb,h)$ expression for proving the replica symmetry breaking in Theorem~\ref{AT line} for \pd~case. Actually, in the RS region for \pd~ $\gD^2$, the infimum should be achieved at the critical point in $\cC(\gb,h)$ instead of some boundary point. This fact can easily be partially checked by comparing $\RS(\gb,h)$ with the replica symmetric solution we obtained from the overlap concentration results in Theorem~\ref{RS solution}. However, to rigorously prove this for the whole region can be a nasty calculus problem. The other way is to prove the convexity of the $\sP_{\RS}$ functional, which is beyond the scope of the current paper. Therefore, we will assume that the infimum in~\eqref{RS} is not achieved at the boundary for Theorem~\ref{AT line}. The following results  were proved in~\cite{BSS19} for $m=2$.
\begin{thm}
	[{\cite{BSS19}*{Theorem 1.1--1.2}}]\label{2spec-res} Restricted to $m=2$ species model, under the assumption
	\begin{align}\label{lamda}
		\gD_{12}^2 =1, \gD_{11}^2\gD_{22}^2 >1 \ \text{and} \ \gl_1 \gD_{11}^2 \ge \gl_2 \gD_{22}^2 ,
	\end{align}
	\begin{enumerate}
		[label=(\roman*)]
		\item If either $h>0$ or
		      \begin{align}\label{con1}
			      \gb^2 <\frac{2}{\gl_1 \gD_{11}^2 + \gl_2 \gD_{22}^2 + \sqrt{(\gl_1\gD_{11}^2-\gl_2\gD_{22}^2)^2+4\gl_1 \gl_2}}
		      \end{align}
		      then $\cC(\gb,h)=\{\vq_{\star}\}$ is a singleton.
		\item Assume $h>0$, let $\vq \in \cC(\gb,h)$, and $ \gc_s:= \E \sech^4(\gb \eta\sqrt{(\gD^{2}\gL\vq)_s}+h)$ for $s=1,2$. If
		      \begin{align}\label{con2}
			      \gb^2 > \frac{2}{\gl_1\gc_1 \gD_{11}^2 + \gl_2\gc_2 \gD_{22}^2 + \sqrt{(\gl_1\gc_1\gD_{11}^2-\gl_2\gc_2\gD_{22}^2)^2+4\gl_1\gl_2\gc_1 \gc_2}}
		      \end{align}
		      then
		      \[\lim_{N\to \infty}F_N < \RS(\gb,h).\]
	\end{enumerate}
\end{thm}

Comparing with the classical SK model, since the RHS of~\eqref{con1} is reduced to the critical $\gb$ for $h=0$ in SK model when $m=1$ and the similar form of RHS of~\eqref{con1} and~\eqref{con2}, one can reasonably guess that~\eqref{con1} gives an RS regime of MSK model, and~\eqref{con2} is the AT line condition. The authors in~\cite{BSS19} proved that~\eqref{con2} indeed gives RSB phase for $m=2$ using the idea in~\cite{Tal11b}*{Chapter~13} by a 1-RSB perturbation of the Parisi formula, and the uniqueness of $\vq_{\star}$ is essential in the proof of Theorem~\ref{2spec-res} part (ii), without that, it is hard to give an accurate description of the AT line.

In this paper, we prove the uniqueness of $\vq_{\star}$ in the general species case under an analogous condition of~\eqref{con1}. To fully generalize Theorem~\ref{2spec-res} in the  $m>2$ case  (\ie~incorporating the case $h>0$), one needs to analyze the uniqueness of the solution to a nonlinear system. Unfortunately, we are not able to prove that. However, assuming $\gD^2$ is \pd~, $\vq_{\star}$ is unique when $h>0$ and the infimum in~\eqref{RS} is not achieved at the boundary, we can prove the AT line condition for general $m$. The idea is still based on the perturbation technique of the Parisi formula. On the other hand, our variance analysis in Section~\ref{sec:cavity} and Section~\ref{sec:varoverlap} suggests the AT line condition should be true from the RS side when $\gD^2$ is \pd.


\subsection{Statement of the main results}\label{s1sec:main}
First, by adapting Lalata's argument, we prove that the MSK model is in the RS regime when $\gb < \gb_0$. As noted in Section~\ref{sec:intro}, our argument holds for \indf~$\gD^2$ and proves the asymptotics of the free energy in the MSK model with \indf~$\gD^2$. To the best of our knowledge, this is the first result dealing with general \indf~$\gD^2$ in the Ising MSK model.

For the rest of the article, we will assume the following:
\begin{ass}\label{ass:0}
	Assume that $\gD^{2}$ is a symmetric and invertible $m\times m$ nonzero matrix with non-negative entries. However, it need not be positive-definite.
\end{ass}

Before stating the RS phase diagram results, we need to generalize that the set $\cC(\gb,h)$ is singleton from $2$-species~\cite{BSS19} to general $m$-species.
We define
\begin{align}\label{betac}
	\gb_{c}:=\rho(\gD^2 \gL)^{-\half}.
\end{align}
where $\rho(A)$ is the \emph{spectral radius} or the largest absolute value of the eigenvalues of $A$. In general, one has $\rho(A) \le \norm{A}$. But it is easy to check that for symmetric $A$, $\rho(A)=\norm{A}$.

\begin{thm}[Uniqueness of solution]\label{thm0}
	Assume that $\gb < \gb_{c}$. Then $\cC(\gb,h)$ is a singleton set, \ie~the system
	\begin{align}\label{syseq}
		q_s = \E \tanh^2(\gb \eta \sqrt{(\gD^2\gL \vq)_s}+h),\qquad s=1,2,\ldots, m
	\end{align}
	has a unique solution, where $\eta\sim \N(0,1)$.
\end{thm}

\begin{rem}
	If we remove the invertiblity of $\gD^2$ in Assumption~\ref{ass:0}, then one has to replace the equations~\eqref{syseq} by setting~\eqref{sys-crit} to 0 and uniqueness holds for the vector $\gD^{2}\gL\vq$.
\end{rem}

The proof is similar to the SK model, which is because the RHS map is a contraction for small $\gb$. For large $\gb$, the contraction is not true anymore. In the SK model with $h>0$, Latala--Guerra lemma~\cite{Tal11a}*{Proposition 1.3.8} tells us $\vq$ is still unique. The proof is based on the monotone property of a nonlinear function and intermediate value theorem. In the MSK model, the analog of Latala--Guerra is not obvious since we are dealing with a system of nonlinear equations~\eqref{syseq}. The authors in~\cite{BSS19} give proof for the $m=2$ case using an elementary approach, but the idea is hard to generalize for $m \ge 3$. Their proof holds only for \pd~$\gD^2$ case. For a particular \indf~MSK model, the deep Boltzmann machine, the Latala-Guerra lemma~\cite{ACM21} has been extended to arbitrary depth with Gaussian random field $h$. For \indf~ $\gD^2$, the following Proposition~\ref{prop:indf-uniq} gives the uniqueness of $\vq_{\star}$ when $m=2$.

\begin{prop}\label{prop:indf-uniq}
	For $m=2$, assume $\gD^2 = \begin{pmatrix}a & b \\c & d \end{pmatrix}$ is \indf. If $h>0$, then the system~\eqref{syseq} has a unique solution for all $\gb>0$.
\end{prop}
This proposition will be used to discuss AT line for some bipartite SK models in Example~\ref{ex:RSB}. The proof of Proposition~\ref{prop:indf-uniq} and Theorem~\ref{thm0} will be given in Section~\ref{sec:RSB}.

Now we define the symmetric \pd~matrix
\begin{align}\label{V}
	\cV:=|\gL^{\sfrac12}\gD^2\gL^{\half}|.
\end{align}
obtained by taking absolute values of all the eigenvalues in the spectral decomposition of the symmetric matrix $\gL^{\sfrac12}\gD^2\gL^{\half}$. We also denote by
\begin{align}
	\ovR_{12} := \vR_{12} -\vq .
\end{align}
the centered overlap vector, where $\vq$ is the unique solution to~\eqref{syseq}. The following two theorems are about the overlap concentration. Here we use $\nu(\cdot) = \E \la \cdot \ra$ to denote expectation w.r.t.~the disorder and the Gibbs measure, one can always check the notation in Section~\ref{notation}.

\begin{thm}[Overlap concentration for general $h$]\label{thm1}
	Assume that $\gb < \gb_0:={\gb_c}/{\sqrt{4\ga}}$, where $\ga =\ga (\gD^{2}):= 1+\ind\{\gD^2 \text{ is indefinite}\}$. For $2\gh<\gb_c^2-4\ga \gb^2 $, we have
	\begin{align*}
		\gn(\exp(\gh N\cP(\ovR_{12})))\le \det(I- (2\gh+4\ga \gb^2 ) \cV )^{-1/2}
	\end{align*}
	where
	\begin{align*}
		\cP(\vx):= \vx^{\T}\gL^{\half}\cV\gL^{\half}\vx.
	\end{align*}
\end{thm}
Theorem~\ref{thm1} says that the overlap vector $\vR_{12}$ concentrates around $\vq$ when $\gb< \gb_0$ for all $h \ge 0$, and its proof is given in Section~\ref{sec2: latala} by adapting Latala's argument. Depending on the definiteness of $\gD^2$, we obtained two different RS regimes. The RS regime for \indf~$\gD^2$ has an extra factor because the control of derivative of interpolated Gibbs average in the proof is somewhat crude. To get a sharper bound, one needs to work much harder, even unknown in SK. However, the next theorem tells us that if $h=0$, one can prove the concentration of overlap up to $\gb_c$, which is true even for \indf~$\gD^2$.
\begin{thm}[Overlap Concentration for $h=0$]\label{ovp no h}
	If $h=0, \gb < \gb_c$, we have
	\begin{align}
		\nu\left(\exp\bigl(\frac18 (\gb_c^2-\gb^2)N \cP(\vR_{12})\bigr)\right) \le K
	\end{align}
	for some constant $K<\infty$ that does not depend on $N$.
\end{thm}
The proof of Theorem~\ref{ovp no h} is given in Section~\ref{concen-h0} and is based on a quadratic coupling argument~\cites{GT02b,Tal11b}. This theorem is expected to give the whole RS regime when $h=0$ (by comparing with the classical SK model), even for \indf~$\gD^2$.

With the control of overlap, we also prove the following.
\begin{thm}[LLN and CLT for free energy]\label{RS solution}
	If $\gb < \gb_0, h\ge 0$, then
	\begin{align}\label{sol}
		\abs{F_N - \RS(\gb,h)} \le \frac K N ,
	\end{align}
	where $K$ is a constant that does not depend on $N$. Moreover, for $h>0$, we have
	\begin{align}\label{clt}
		N^{\half}\left( F_N - \RS(\gb,h) \right) \Rightarrow \N(0,b(\gb,h)) \quad \text{as} \quad N \to \infty,
	\end{align}
	where
	\begin{align*}
		\RS(\gb,h) & = \log 2 + \sum_{s=1}^m\gl_s \E \log \cosh(\gb \eta \sqrt{(\gD^2\gL \vq)_s} +h) + \frac{\gb^2}{4}\cQ(\vone -\vq), \\
		b(\gb,h)   & :=\sum_{s=1}^m \gl_{s}\var(\log\cosh(\gb \eta \sqrt{(\gD^2\gL \vq)_s}+h)) - \frac{\gb^2}2 \cQ(\vq).
	\end{align*}
	and $\cQ(\vx) := \vx^{\T}\gL\gD^2\gL \vx$ denotes the quadratic form of $\gL \gD^2\gL$.
\end{thm}
The first part of Theorem~\ref{RS solution} gives the RS solution of MSK model for general $\gD^2$. It's easy to check that $\text{RS}(\gb,h) $ here has the same form as $\text{RS}(\gb,h)$ in~\eqref{RS} derived from the Parisi formula in the \pd\ $\gD^{2}$ case. However, in the bipartite case the critical points in $\cC(\gb,h)$ are saddle points, so the formula in~\eqref{par} is not true anymore. However, a modified Parisi formula is likely to be true for $\gD^2$ \indf~at least in the RS regime (see Example~\ref{ex:RSB} for further discussion). The second part is a Central Limit Theorem for the free energy, which holds at the corresponding RS regime depending on $\gD^2$. The proof of Theorem~\ref{RS solution} is in Section~\ref{clt-free-engy}. For the CLT of free energy in classical SK model, see~\cites{GT02,ALR87,Tin05,CN95} and references therein.

\begin{rem}
	The proof of CLT for the free energy in the $\gb<\gb_c, h=0$ case, is similar to the classical SK model, see~\cite{ALR87} and~\cite{Tal11b}*{Chapter 11}. In~\cite{ALR87}, the authors proved the CLT using cluster expansion approach and later~\cite{Tal11b} reproved it using moment method. Both the cluster expansion approach and moment method are highly nontrivial for MSK model. However, the characteristic function approach we used in our paper, can be generalize for  the MSK model with $\gb<\gb_{c}, h=0$.  Since we already have a good control of overlap in Theorem~\ref{ovp no h}, the next question is to determine the asymptotic mean and variance of free energy. By the species-wise cavity approach developed in Section~\ref{sec:cavity} and Section~\ref{sec:varoverlap}, we evaluate
	\begin{align}
		c_N(\gb) & := N\left(\log 2 +\frac 1 4 \gb^2 \cQ(\vone)\right) + \frac{1}{4} \log \det(I-\gb^2\gD^2\gL), \\
		b(\gb)   & :=\frac 1 2 \left(-\log \det(I-\gb^2\gD^2\gL)-\gb^2\tr(\gD^2\gL)\right).
	\end{align}
	Then $\log Z_N(\gb) - c_N(\gb) \Rightarrow \N(0,b(\gb))$ as $N \to \infty$ can be carried over similarly as in~\cite{Tal11b}*{Chapter 11}, but we do not pursue that here.
\end{rem}

Next, we generalize the classical cavity approach to a species-wise cavity method to analyze the fluctuation of the overlap vector. We choose a species $s \in \{1,2,\ldots,m\}$ first with probability ${\gl_s}$, then select a spin uniformly in that species to be decoupled and finally compare this decoupled system with the original one. The complete description is given in Section~\ref{sec:cavity}. In~\cite{Pan15}, Panchenko introduced another cavity idea when using the Aizenman--Sims--Starr scheme to prove the lower bound of the Parisi formula. It adds/decouples $k>1$ spins simultaneously, and those $k$ spins are distributed into $m$ species according to the ratios stored in $\gL$. It might be possible to do the second-moment computation along with this cavity idea, but it will be more technically challenging since one needs to control $k>1$ spins.

For $(\gb,h)$ in the RS region proved in Theorem~\ref{thm1} and~\ref{ovp no h} with $\vq$ being the unique solution to~\eqref{syseq}, define
\[
	\hat{q}_{s}:=\E\tanh^{4}(\gb\eta\sqrt{(\gD^{2}\gL\vq)_{s}}+h),\qquad s=1,2,\ldots,m.
\]
Define the $m\times m$ diagonal matrices $\gC, \gC', \gC''$, whose $s$-th diagonal entries are respectively given by
\begin{align*}
	\gc_s = 1-2q_{s}+\hat{q}_s          & =\E\sech^{4}(\gb\eta\sqrt{(\gD^{2}\gL\vq)_{s}}+h), \\
	\gc'_s = 1 -4q_{s}+3\hat{q}_s\qquad & \text{and} \qquad
	\gc''_s = 2q_s +q_s^2 -3\hat{q}_s \qquad \text{ for } s=1,2,\ldots,m.
\end{align*}
Note that $-1/3\le \gc_{s}'\le \gc_{s}\le 1$ but $\gc'_{s}$ can be negative. For large $h$, one can easily check that $\gc_{s}' < -\gc_{s}$. This fact will be used when we discuss the AT line condition in Example~\ref{ex:RSB}. Let
\begin{align}\label{Ui}
	U(0):=\nu(\ovR_{12}\ovR_{34}^{\T}),\quad U(1):=\nu(\ovR_{12}\ovR_{13}^{\T})\text{ and }\quad U(2):=\nu(\ovR_{12}\ovR_{12}^{\T})
\end{align}
denote the variance-covariance matrices of the overlap vectors. Note that, all the matrices $U(i)$'s are \pd. Using the generalized cavity method we prove the following.

\begin{thm}\label{thm:lyap}
	For $\gb< \gb_{0}$ or $\{\gb<\gb_{c}, h=0\}$, the matrices
	\begin{align*}
		\hat{U}(0) & := -3U(0)+2U(1),     \\
		\hat{U}(1) & := 3U(0)-4U(1)+U(2), \\
		\hat{U}(2) & := U(0)-2U(1)+U(2)
	\end{align*}
	satisfy the following equations
	\begin{align}\label{lyap}
		\begin{split}
			\sym\bigl( \bigl(I-\gb^2\gC\gD^{2}\gL\bigr) \cdot \hat{U}(2)\bigr) &= \frac1N\cdot \gC\gL^{-1} + \mfR,\\
			\sym\bigl( \bigl(I-\gb^2\gC'\gD^{2}\gL\bigr) \cdot \hat{U}(1)\bigr) &= \frac1N\cdot \gC'\gL^{-1} + \mfR,\\
			\text{ and }  \sym\bigl( \bigl(I-\gb^2\gC'\gD^{2}\gL\bigr) \cdot \hat{U}(0)\bigr) &= \gb^{2} \sym\bigl( \gC''\gD^{2}\gL\cdot \hat{U}(1)\bigr) + \frac1N\cdot \gC''\gL^{-1} + \mfR,
		\end{split}
	\end{align}
	where $\sym(A)=(A+A^{\T})/2$ and $\mfR$ is some $m\times m$ symmetric matrix with
	\[ \max_{1\le p,q\le m}|\mfR_{p,q}|\le K\biggl(N^{-3/2} + \sum_{s=1}^{m}\nu(\abs{R_{12}^s-q_s}^3)\biggr)\]
	for some constant $K$ independent of $N$.
\end{thm}
Unfortunately, we do not have a nice interpretation of the coefficients appearing in the system of linear equations in Theorem~\ref{thm:lyap}. This is actually related to an open problem by Talagrand, where he asked to identify the underlying algebraic structure (see~\cite[Research Problem 1.8.3]{Tal11a}).

Before we discuss the Theorem~\ref{thm:lyap}, let's recall the definition of stable matrices.
\begin{defn}\label{def:stable}
	Square matrix $A$ is \emph{stable} if all the eigenvalues of $A$ have strictly negative real part.
\end{defn}

Note that the linear equations in~\eqref{lyap} are examples of continuous Lyapunov equation of the form $AX+XA^{\T}+Q=0$ where $A, Q$ are symmetric matrices. Those equations also appear in control theory (see \eg~\cite{AM07book}). Moreover, the existence and uniqueness of the solution are equivalent to the matrix $A$ being stable. It is an exciting question to connect equations~\eqref{lyap} with an appropriate control problem. In our case, solving the system of equations given in~\eqref{lyap}, we get the asymptotic variance of overlap.

\begin{thm}[Asymptotic variance of overlap vector]\label{thm:varoverlap}
	For $\gb< \gb_{0}$ or $\{\gb<\gb_{c}, h=0\}$, we have
	\begin{align*}
		N \cdot U(i) \to \gS(i) \quad \text{as} \quad N \to \infty
	\end{align*}
	where $\gS(i), i=0,1,2$ satisfies the following
	\begin{align}\label{eq:solve-lya}
		\begin{split}
			\gS(0)-2\gS(1)+\gS(2) &= \gC(I - \gb^2\gD^{2}\gL\gC)^{-1}\gL^{-1},\\
			3\gS(0)-4\gS(1)+\gS(2) &= \gC'(I - \gb^2\gD^{2}\gL\gC')^{-1}\gL^{-1}.
		\end{split}
	\end{align}
\end{thm}

\begin{rem}
	$\gS(i), i=0,1,2$ also solve a third equation~\eqref{eq:var-overlap-3} besides~\eqref{eq:solve-lya}, but it does not possess a simple form, see the details in Section~\ref{sec:varoverlap}.
\end{rem}

The asymptotic variance of overlap is given as solution to the linear system in~\eqref{lyap}. We recall the definition of \emph{stable} in the Definition~\ref{def:stable}. To get a unique solution from the system~\eqref{lyap}, one needs the matrices $\gb^2\gC\gD^{2}\gL-I, \gb^2\gC'\gD^{2}\gL-I$ to be stable, \ie\
\begin{align}\label{cond1:solve sys}
	\gb^2\max_{\gl \in \text{spec}(\gC\gD^2\gL)} \Re(\gl)<1 \quad \text{and} \quad \gb^2\max_{\gl \in \text{spec}(\gC'\gD^2\gL)} \Re(\gl) <1.
\end{align}

Note that $\gc_s >0 $ for all $s$ and $\gC\gD^2\gL$ is similar to a symmetric matrix. Thus all eigenvalues of $\gC\gD^2\gL$ are real, and by Perron--Frobenius theorem
\[
	\max_{\gl \in \text{spec}(\gC\gD^2\gL)} \Re(\gl)=\rho(\gC\gD^2\gL)
\]
and the condition~\eqref{cond1:solve sys} is equivalent to
\begin{align}\label{cond2:solve sys}
	\gb^2\max\{\rho(\gC\gD^2\gL),\max_{\gl\in \text{spec}(\gC'\gD^2\gL)} \Re(\gl)\}<1.
\end{align}

For \pd~$\gD^2$, the matrix $\gC'\gD^2\gL$ is similar to a symmetric matrix and has real eigenvalues. Moreover, using $\gc'_s \le \gc_s $, one can get
\begin{align}\label{cond3:psd solving}
	\max_{\gl\in \text{spec}(\gC'\gD^2\gL)} \Re(\gl) \le \rho(\gC\gD^2\gL).
\end{align}
Thus the results in Theorem~\ref{thm:varoverlap} will not be true unless $\gb^2<\rho(\gC\gD^2\gL)^{-1}$. This indicates that $\gb^2 \rho(\gC\gD^2\gL)=1$ is the AT line condition from the RS side for \pd~$\gD^{2}$.
To rigorously prove that $\gb^2< \rho(\gC\gD^2\gL)^{-1}$ gives RS phase is still open even in classical SK model. In~\cite{JT17}, the authors proved it except for a bounded region close to the phase boundary by using the Parisi formula. Recently, Bolthausen~\cites{Bol14,Bol19} developed an iterative TAP approach to SK model, where he also discussed some possible ideas to prove $\lim_{N\to \infty} F_N = \text{RS}(\gb,h)$ below the AT line.
For $h=0$, when $\gb^2< \gb_{c}^{2}=\rho(\gD^2\gL)^{-1}$ and $\gD^{2}$ is general, it is easy to check that condition~\eqref{cond1:solve sys} is satisfied, more details can be found in Section~\ref{sec:varoverlap}.

The following Theorem~\ref{AT line} states the RSB condition when $\gD^2$ is \pd, under the assumption that $\cC(\gb,h)$ is a singleton for $h>0$. Without this, it is not easy to give an accurate description of the AT line. We assume $\gD^2$ is \pd, as the proof is based on perturbation of the Parisi formula known only in that case in~\cite{Pan15}. Recall that we also need to assume that the infimum in~\eqref{RS} can only be achieved in $\cC(\gb,h)$ but not on the boundary.
\begin{thm}[RSB condition]\label{AT line}
	Assume that $\gD^2$ is \pd~and for some $\gb,h>0$, $\cC(\gb,h)$ is a singleton. We further assume that the infimum in~\eqref{RS} is not achieved on the boundary. Let $\vq$ be the unique critical point in $\cC(\gb,h)$. Define
	$
		\gC := \diag(\gc_1, \gc_2, \ldots, \gc_m)
	$
	where
	$
		\gc_s := \E \sech^4(\gb \eta\sqrt{Q_1^s}+h) \ \text{for} \ s=1, 2,\ldots, m.
	$
	If
	$
		\gb>\rho(\gC\gD^2\gL)^{-\half}
	$, then
	\[
		\lim_{N\to \infty}F_N < \RS(\gb,h).
	\]
\end{thm}
Note that
\begin{align}\label{ATline}
	\gb_{\textsc{AT}}(h):=\rho(\gC\gD^2\gL)^{-\half}
\end{align}
seems to be the AT line in the MSK model for \pd~$\gD^{2}$, and when $h=0$, $\gb_{\textsc{AT}}(0) = \gb_c$ gives the critical temperature defined in~\eqref{betac}, see the left phase diagram picture of Figure~\ref{fig:phase}. When $h=0, \gb>\gb_c$, there seems to be at least one non-zero solution in $\cC(\gb,0)$. However, to prove RSB by the perturbation argument is still technically challenging.
For the general MSK model, here is a class of examples with a non-zero solution for $h=0,\gb>\gb_{c}$.

In the classical SK model, one can easily check that for $h=0,\gb>1$, there exists a non-zero solution $\hat{q}(\gb)\in (0,1)$ to  the equation $q=\E \tanh^2(\gb\eta \sqrt{q})$.  Take $\gD^2, \gL$ such that $\vone$ is an eigenvector of $\gD^{2}\gL$ with associated eigenvalue $\theta$. Then  $\gb_c = \rho(\gD^2\gL)^{-\half} = \theta^{-\half}$ and for $\gb>\gb_c$, we have $ \gb \sqrt{\theta}>1$. It follows that $\vq=\hat{q}(\gb \sqrt{\theta})\vone$ is a non-zero solution to the equations
\begin{align*}
	q_s= \E \tanh^2(\gb\eta \sqrt{(\gD^{2}\gL\vq)_s}) \text{ for }  s=1,2,\cdots, m.
\end{align*}

Finally, we point out that the AT line condition for $\gD^2$ \indf~seems to be  more complicated than the \pd~case. Because for \indf~$\gD^2$, the inequality~\eqref{cond3:psd solving} could fail depending on the values of $(\gb,h)$. For fixed $\gb>\gb_c$ and large $h$, we observe $\max_{\gl\in \text{spec}(\Gamma'\gD^2\gL)} \Re(\gl) \ge \rho(\gC\gD^2\gL)$ in some examples. However, it is not obvious whether in the $\gD^2$ \indf~case, the AT line is given by $\gb^2\rho(\gC\gD^2\gL)=1$ (see from the right picture in Figure~\ref{fig:phase}).

\begin{figure}[ht]
	\centering
	\includegraphics[scale=0.75]{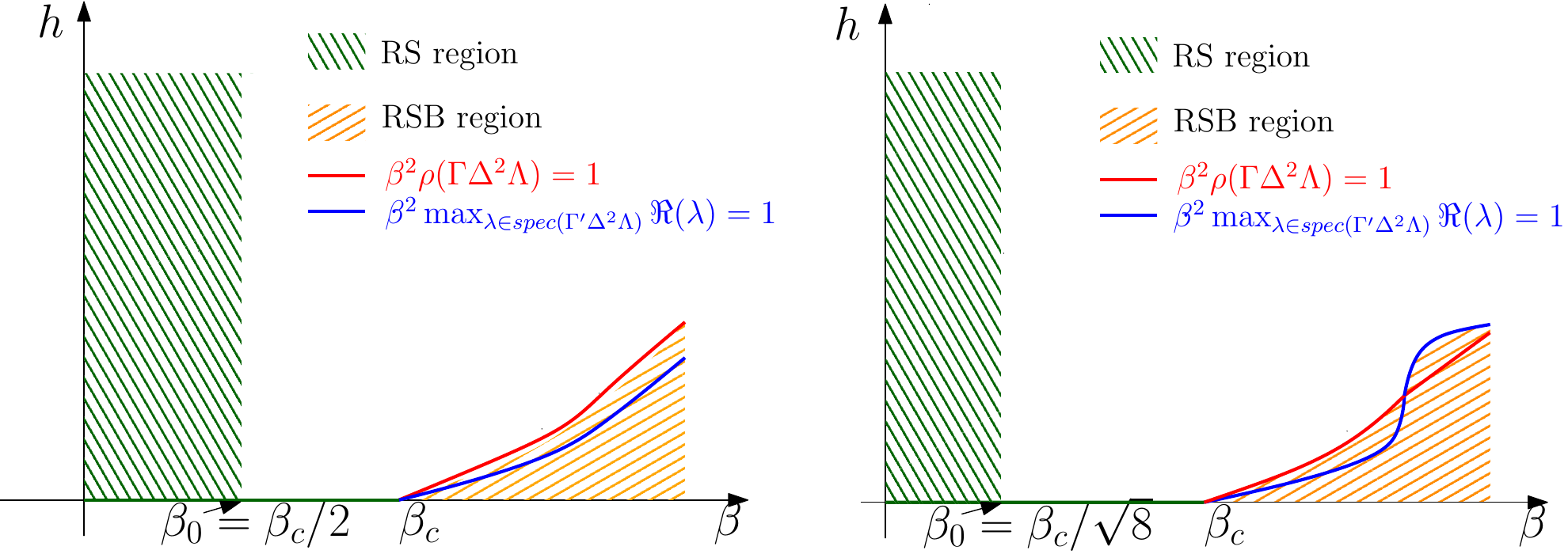}
	\caption{Phase diagram of MSK model: the left and right one are expected for \pd~ and \indf~$\gD^2$ respectively. The green shaded region corresponds to RS regime proved in Theorem~\ref{thm1} and~\ref{ovp no h}. The expected RSB phase is labelled in both cases. In both pictures, the red curve and blue curve correspond to $\gb^2 \rho(\gC\gD^2\gL)=1$ and $\gb^2\max_{\gl\in \text{spec}(\gC'\gD^2\gL)} \Re(\gl)=1$ respectively. The right picture for \indf~$\gD^2$ is based on the evidence from some examples, see Example~\ref{ex:RSB}.}
	\label{fig:phase}
\end{figure}

\begin{ex}\label{ex:RSB}
	Take $\gL =
		\begin{pmatrix}
			1/2 & 0   \\
			0   & 1/2
		\end{pmatrix}
	$ and $\gD^2 =
		\begin{pmatrix}
			0 & 1 \\
			1 & 0
		\end{pmatrix}$. This is the bipartite SK model, an example of \indf~MSK model. For $h>0$, by Proposition~\ref{prop:indf-uniq}, the following system
	\[
		\begin{cases}
			q_1 = \E \tanh^2(\gb \eta \sqrt{q_2/2} +h) \\
			q_2 = \E \tanh^2(\gb \eta \sqrt{q_1/2} +h)
		\end{cases}
	\]
	has a unique solution $q_1=q_2 = q= \E \tanh^2(\gb \eta \sqrt{\frac{q}{2}}+h)$. Similarly, $\gc_1 = \gc_2 = \gamma$ and $\gc_1' = \gc_2' =\gamma'$.

	Then we have
	\begin{align*}
		 & \rho(\gC\gD^2\gL) = \frac{\gamma}{2} \quad \text{and} \quad \max_{\gl\in \text{spec}(\Gamma'\gD^2\gL)} \Re(\gl)= \rho(\gC'\gD^2\gL) =\frac{\abs{\gamma'}}{2}.
	\end{align*}
	We know that $\gc \ge \gc', \gc>0$ always holds, but $\abs{\gamma'} \le \gamma$ is not always true (for large $h$), and it depends on the strength of external field $h$. The suggested AT line condition is $\gb^2 \max\{\gc, |\gc'|\}<1$.

	For the functional $\sP_{\RS}(\vq)$, the Hessian w.r.t.~$\vq$ is a positive multiple of
	$$H= \gD^2 - \gb^2 \gL\gC' =
		\begin{pmatrix}
			- \gb^2 \gc'/2 & 1 \\ 1 & - \gb^2\gc'/2
		\end{pmatrix}.$$
	$H$ is an indefinite matrix in the replica symmetric region, as the eigenvalues of $H$ are given by
	$ - \gb^2 \gc'/2 \pm 1 $
	and the RS region seems to be $ \beta^2 \max( \gamma, |\gamma'| ) /2 < 1 $.
	So, although the Parisi formula evaluated at the stationary point gives the limit for the free energy (proved for small $\beta$), it is not the minimizer of the Parisi formula directly taken from the positive-definite case.
\end{ex}

In general, we conjecture the RSB condition for \indf~$\gD^2$ as follows.
\begin{conj}\label{AT conj}
	For general indefinite $\gD^2$, if $$\gb^2\max\{\rho(\gC\gD^2\gL),\max_{\gl\in \text{spec}(\Gamma'\gD^2\gL)} \Re(\gl)\}>1,$$ the system is in RSB phase.
\end{conj}

The above conjecture seems far to solve since the classical idea is based on perturbation of the Parisi formula and the Parisi formula in general \indf~case is still in mystery. Besides that, we also need to understand the solution to~\eqref{syseq} for \indf~$\gD^2$. To prove the AT line condition from the RS side for \indf~$\gD^2$ is even more challenging.

In order to illustrate the generality and correctness of our results, we look at some models studied in other papers as particular cases in our setting.
\begin{ex}
	Take $\gL =
		\begin{pmatrix}
			\gl_1 & 0     \\
			0     & \gl_2
		\end{pmatrix}
	$ and $\gD^2$ as in~\eqref{lamda}, then
	$$
		\gb_c= \rho(\gD^2 \gL) =\frac12\left(\gl_1 \gD_{11}^2 + \gl_2 \gD_{22}^2 + \sqrt{(\gl_1\gD_{11}^2-\gl_2\gD_{22}^2)^2+4\gl_1 \gl_2}\right).
	$$
	Similarly, one can check
	$$
		\gb_{AT}(h) = \rho(\gC\gD^2\gL)= \frac12\left(\gl_1\gc_1 \gD_{11}^2 + \gl_2\gc_2 \gD_{22}^2 + \sqrt{(\gl_1\gc_1\gD_{11}^2-\gl_2\gc_2\gD_{22}^2)^2+4\gl_1\gl_2\gc_1 \gc_2}\right).
	$$
	Theorem~\ref{thm0} and~\ref{AT line} extend the critical temperature~\eqref{con1} and AT line condition~\eqref{con2} to the general MSK model.
\end{ex}

\begin{ex}
	Take $\gL =
		\begin{pmatrix}
			\gl_1 & 0     \\
			0     & \gl_2
		\end{pmatrix}
	$ and $\gD^2 =
		\begin{pmatrix}
			0          & 1 \\
			\theta^{2} & 0
		\end{pmatrix}$. For $\theta=1$, this corresponds to the bipartite SK model. In~\cite{BL17} the authors study the BSSK model on the sphere with $h=0$ and compute $\gb_c = (\gl_1\gl_2)^{-1/4}$. In our case, we have $\gb_c = \rho(\gD^2\gL)^{-\half}$ in~\eqref{betac} and by a simple computation, one gets $\rho(\gD^2\gL)^{-\half}=(\gl_1\gl_2\theta^{2})^{-1/4}$, which agrees with $\gb_c$ in~\cite{BL17}.
\end{ex}


\subsection{Notations} \label{notation}
\begin{enumerate}
	\item $\rho(A)$ is the \emph{spectral radius} or the largest absolute value of the eigenvalues of $A$.
	\item $\norm{A}$ is the operator norm of $A$.
	\item $\gD^2 = (\!(\gD_{s,t}^2)\!)_{s,t=1}^m$ is the covariance matrix of Gaussian disorder.
	\item $\gL:=\diag(\gl_1,\gl_2,\ldots, \gl_m)$ is the diagonal matrix, whose diagonal entries is the ratio of the spins in each species.
	\item $\cV:=|\gL^{\sfrac12}\gD^2\gL^{\half}|$ which is obtained by taking absolute values of all the eigenvalues in the spectral decomposition of the matrix $\gL^{\sfrac12}\gD^2\gL^{\half}$.
	\item $\vR_{12} = (R^{(1)}_{12}, R^{(2)}_{12}, \ldots, R^{(m)}_{12})$ denotes the overlap vector.
	\item $\ovR_{12} := \vR_{12} -\vq$ is the centered overlap vector, and $\vq$ is defined through~\eqref{syseq}.

	\item $Q^s :=(\gD^2\gL\vq)_s \text{ for } s=1,2, \ldots, m$.

	\item $ \Gamma := \diag(\gc_1,\gc_{2}, \ldots, \gc_m)$ where $\gc_s := \E \sech^4(\gb \eta\sqrt{Q^s}+h) \text{ for }  s=1,2, \ldots, m$.

	\item $\la f(\mvgs) \ra := \sum_{\mvgs\in \gS_N} f(\mvgs ) \cdot {\exp(H_N(\mvgs))}/{Z_N} $ denotes the Gibbs average of function $f$ on $\gS_N$.

	\item $\gn_t(\cdot)$ denotes the expectation w.r.t.~the Gibbs randomness with interpolated Hamiltonian and disorder.

	\item $\cQ(\vx) = \vx^{\T}\gL\gD^2\gL \vx$ denotes the quadratic form associated with the symmetric matrix $\gL\gD^2\gL$. For a general symmetric matrix $A$, we will use $\cQ(\vx,A)$ for $\vx^\T A \vx$. We also write $\cP(\vx)=\cQ(\vx, \gL^{\half}\cV\gL^{\half})$.
\end{enumerate}

\subsection{Roadmap}\label{roadmap} The paper is structured as follows. Section~\ref{sec2: latala} is mainly about the proof of Theorem~\ref{thm1}, where the smart path interpolation argument is applied in the MSK model to prove the concentration of overlap at high temperature. The argument also holds for indefinite $\gD^2$, while in this case, the concentration happens in a different regime. Section~\ref{clt-free-engy} is mainly about the proof of Theorem~\ref{RS solution}, where we prove the RS solution of the MSK model and a Central Limit Theorem for the free energy with a non-zero external field. In Section~\ref{concen-h0}, we study the MSK model without an external field, where the proof of Theorem~\ref{ovp no h} is presented. A Central Limit Theorem for the free energy with zero external field can also be proved in a similar way as in~\cite{ALR87}, but we do not pursue this direction here. In Section~\ref{sec:cavity}, we develop a generalized cavity method to study the MSK model. Using second-moment computations, we derive a linear system of overlap functions and solve it using linear algebraic methods to get the variance-covariance structure of overlap vectors in Section~\ref{sec:varoverlap}. Basically, Section~\ref{sec:cavity} and~\ref{sec:varoverlap}  contain the proofs of Theorem~\ref{thm:lyap} and~\ref{thm:varoverlap} respectively. In Section~\ref{sec:RSB}, for \pd~MSK model, we give the proof of Theorem~\ref{AT line} by using the perturbation argument of the Parisi formula. Finally Section~\ref{sec:openqn} contains discussion and some further questions.

\section{Concentration of overlap in MSK}\label{sec2: latala}
Originally Latala's argument~\cite{La02} in the SK model is used to prove the concentration of overlap in part of the RS regime. The idea is based on Guerra's interpolation of two different spin glass models, one of which is the fully decoupled model, \ie~the associated Gibbs measure is a product measure, while the other model is the standard SK model of our interest. With the concentration of overlap in the decoupled model, one can obtain the concentration of overlap in the SK model by controlling the derivative of the interpolated model. For the details of this classical story, see Section 1.3 and 1.4 in~\cite{Tal11a}. This section employs a similar idea in the MSK model to prove the concentration of overlap vectors. With this concentration result, we also prove a part of the RS phase diagram.

Following Guerra's interpolation, given two independent centered Gaussian vectors $\mvu=(u_{\mvgs}), \mvv=(v_{\mvgs})$ indexed by $\mvgs\in\gS_{N}$, consider the interpolation given by $u_{\mvgs}(t) = \sqrt{t}\cdot u_{\mvgs}+\sqrt{1-t}\cdot v_{\mvgs}$, for $t \in [0,1]$. Suppose $F$ is a twice differentiable function allowing us to use Gaussian integration by parts. Taking $\phi(t) := \E F(\mvu(t))$ and using Gaussian integration by parts we get the following.
\begin{prop}
	For $\mvgs, \mvgt \in \gS_N$,
	\[ \phi^{\prime}(t) = \sum_{\mvgs,\mvgt}U(\mvgs,\mvgt) \E\left( \frac{\partial^2F}{\partial x_{\mvgs}x_{\mvgt}}(\mvu(t))\right) \]
	where $U(\mvgs,\mvgt) := \frac{1}{2}(\E u_{\mvgs} u_{\mvgt} -\E v_{\mvgs} v_{\mvgt} ) $.
\end{prop}
Given $(\eta_{i})_{i=1,2,\ldots,N}$ i.i.d.~standard Gaussian random variables and independent of $(g_{ij})_{1\le i<j\le N}$, we take
\[ 
u_{\mvgs} = \frac{\beta}{\sqrt{N}} \sum_{i<j}g_{ij}\gs_i\gs_j, \quad v_{\mvgs} =\beta \sum_{s=1}^{m}\sqrt{(\gD^{2}\gL\vq)_{s}}\cdot \sum_{i \in I_s}\eta_{i} \gs_i, \quad w_{\mvgs} = \exp\bigl(h\sum_{i}\gs_i\bigr) \]
and $F(\mvx) = \frac1N \log Z_N$ to be the free energy, where $Z_N=\sum_{\mvgs}w_{\mvgs} e^{x_{\mvgs}}$. In this case, the interpolated Hamiltonian is
\begin{align}\label{inter-Hamilt}
	H_{N,t}(\mvgs) := u_t(\mvgs)= \sqrt{t}u_{\mvgs} + \sqrt{1-t}v_{\mvgs}.
\end{align}
Moreover, we have
\begin{align*}
	\E u_{\mvgs} u_{\mvgt} & = \frac{N\beta^2}{2}\left(\sum_{s,t}\lambda_s R^s\gD_{s,t}^2\lambda_tR^t-\frac{1}{N}\sum_s\gl_s\gD_{s,s}^2\right), \\
	\E v_{\mvgs} v_{\mvgt} & = N\beta^2\sum_s \gl_s R^s Q^s = N\beta^2\sum_{s,t} \gl_s R^s \gD_{s,t}^2 \gl_{t}q_{t}.
\end{align*}
In particular,
\begin{align*}
	U(\mvgs,\mvgt)  & = \frac{1}{2}(\E u_{\mvgs} u_{\mvgt} -\E v_{\mvgs} v_{\mvgt}) = \frac {N\beta^2} 4 \left(\cQ(\ovR)- \cQ(\vq) -\frac{1}{N}\sum_s \gl_s\gD_{s,s}^2\right) \\
	\text{ and }
	U(\mvgs, \mvgs) & =\frac {N\beta^2}4 \left(\cQ(\vone-\vq)- \cQ(\vq) -\frac{1}{N}\sum_s \gl_s\gD_{s,s}^2\right).
\end{align*}
By a simple computation, we have
\[ \frac{\partial F}{\partial x_{\mvgs}} = \frac1N\cdot w_{\mvgs} e^{x_{\mvgs}}/Z_N, \quad \frac{\partial^2 F}{\partial x_{\mvgs} \partial x_{\mvgt}} = \frac1N\cdot (\ind_{\mvgs=\mvgt} - w_\gt e^{x_{\mvgt}}/Z_N)\cdot w_{\mvgs} e^{x_{\mvgs}}/Z_N \]
and we use this to rewrite the formula of $\phi^{\prime}(t)$,
\begin{align}\label{free-ener deriv}
	\phi^{\prime}(t) & = \frac{1}{N} \E(\la U(\mvgs,\mvgs) \ra_t - \la U(\mvgs,\mvgt) \ra_t) = \frac{\gb^2}{4}\left(\cQ(\vone-\vq) - \gn_t(\cQ(\ovR_{12}))\right).
\end{align}
In the following section, we will prove that $\vR_{12}$ is concentrated around $\vq$ at some high temperature regime, \ie\ the term $\gn_t(\cQ(\ovR_{12}))$ is very small, and this concentration property will enable us to prove the RS solution in MSK model.

First, we will prove that the quantity $\gn_t(\cQ(\ovR_{12}))$ is small. The basic idea is to prove some exponential moments is non-increasing in $t$, then by controlling $\gn_0(\cQ(\ovR_{12}))$, we can analyze the other side: $\gn(\cQ(\ovR_{12}))$ along the interpolation path. In order to prove the Theorem~\ref{thm1}, we need to study the property of $\gn_t^{\prime}$, \ie~the derivative of $\gn_t$ with respect to $t$. The following lemma gives the expression of $\gn_t^{\prime}$.
\begin{lem}
	[{~\cite{Tal11a}*{Page~33}}] If $f$ is a function defined on $\gS_N^n$, then
	\begin{align*}
		\gn_t^{\prime}(f) & = \sum_{1 \le  l,l^\prime  \le  n} \gn_t(U(\mvgs^l,\mvgs^{l^{\prime}})f) + n(n+1)\gn_t(U(\mvgs^{n+1},\mvgs^{n+2})f) \\
		                  & \qquad\qquad\qquad-2n \sum_{l  \le  n} \gn_t(U(\mvgs^l,\mvgs^{n+1})f) - n\gn_t(U(\mvgs^{n+1},\mvgs^{n+1})f).
	\end{align*}
\end{lem}
Note that $\gn_t(U(\mvgs^l,\mvgs^{l})f) =\gn_t(U(\mvgs^{n+1},\mvgs^{n+1})f) $, therefore the above formula can be simplified one more step as follows:
\begin{align}\label{eq:gen-derivative}
	\gn_t^{\prime}(f) & = \sum_{1 \le  l <l^\prime  \le  n} 2\gn_t(U(\mvgs^l,\mvgs^{l^{\prime}})f) + n(n+1)\gn_t(U(\mvgs^{n+1},\mvgs^{n+2})f) -2n \sum_{l  \le  n} \gn_t(U(\mvgs^l,\mvgs^{n+1})f).
\end{align}
When $n=2$, simplifying we have
\begin{align}\label{derivative}
	\gn_t^{\prime}(f) & = \frac{N\gb^2}{2}\bigl(\gn_t(\cQ(\ovR_{12})f) - 2\sum_{l  \le  2}\gn_t(\cQ(\ovR_{l,3})f) + 3 \gn_t(\cQ(\ovR_{34})f)\bigr).
\end{align}
Next, we prove a useful lemma comparing the quadratic forms of overlap under $\gn_t'(\cdot)$.
\begin{lem}\label{holder lem}
	Consider $\gr >0$, for $l, l' \in \bN$, we have
	\begin{align*}
		\abs{ \gn_t (\cQ(\ovR_{\ell\ell'})\exp(\gr N\cP( \ovR_{12}))) } & \le \gn_t(\cP(\ovR_{12})\exp(\gr N\cP( \ovR_{12})))
	\end{align*}
	where $\cP(\vx) =\vx^{\T}\gL^{\half}\cV\gL^{\half}\vx$.
\end{lem}
\begin{proof}
	[Proof of Lemma~\ref{holder lem}.] Let $i:= \abs{\{\ell,\ell'\} \cap \{1,2\}}$, we prove the above inequality case by case. For $i=2$, we need to prove
	\[ \gn_t(\cQ(\ovR_{12})\exp(\gr N\cP( \ovR_{12}))) \le \gn_t(\cP(\ovR_{12})\exp(\gr N\cP( \ovR_{12}))) \]
	the inequality is obvious by the fact that $|\cQ(\vx)|\le \cP(\vx)$. For $i=1$ and $i=0$, it can be proved by H\"older inequalities. We provide a proof for the case $i=1$. First, the general form of H\"older inequality w.r.t.~$\gn_t$ is:
	\[ \gn_t(f_1f_2)  \le  \gn_t(f_1^{\gt_1})^{1/\gt_1}\gn_t(f_2^{\gt_2})^{1/\gt_2} \quad \text{for} \quad f_1, f_2  \ge  0, \quad \frac{1}{\gt_1}+ \frac{1}{\gt_2} =1 \]
	Without loss of generality, we assume $\ell=1$ and $\ell'=3$, then one needs to prove
	\[ \abs{\gn_t(\cQ(\ovR_{13})\exp(\gr N\cP( \ovR_{12})))} \le \gn_t(\cP(\ovR_{12})\exp(\gr N\cP( \ovR_{12}))) \]
	Since
	\begin{align}\label{exp expans}
		\abs{\gn_t(\cQ(\ovR_{13})\exp(\gr N\cP( \ovR_{12}))) } & \le \gn_t(\cP(\ovR_{13})\exp(\gr N\cP( \ovR_{12})))\notag           \\
		                                                       & = \sum_k \frac{(\gr N)^k}{k!}\gn_t(\cP(\ovR_{13})\cP( \ovR_{12})^k)
	\end{align}
	Applying H\"{o}lder inequality with $\gt_1 = k+1, \gt_2 = (k+1)/k$, we get
	\begin{align*}
		\gn_t(\cP(\ovR_{13})\cdot \cP( \ovR_{12})^k) & \le \gn_t(\cP(\ovR_{13})^{k+1})^{1/(k+1)} \cdot \gn_t(\cP( \ovR_{12})^{k+1})^{k/(k+1)} = \gn_t(\cP(\ovR_{12})^{k+1})
	\end{align*}
	the last equality is by the symmetry among replicas. Combining with the exponential expansion~\eqref{exp expans}, we proved the case $i=1$. The case $i=0$ follows in a similar fashion.
\end{proof}
\begin{rem}\label{non-neg remark}
	In the Lemma~\ref{holder lem}, if $\gD^2$ is non-negative definite, then $\cP(\ovR_{12}) =\cQ(\ovR_{12})$, thus the similar inequalities of Lemma~\ref{holder lem} also hold by replacing $\cP(\ovR_{12})$ with $\cQ(\ovR_{12})$ and the proof is similar but simpler.
\end{rem}

\begin{cor}\label{derivat bound}
	If $\gk >0 $, then
	\begin{align*}
		\gn_t^{\prime}(\exp(\gk N \cP(\ovR_{12}))  \le  2 \ga N\gb^2\gn_t(\cP(\ovR_{12})\exp(\gk N \cP(\ovR_{12}))),
	\end{align*}
	where $\ga = 1+\ind\{\gD^2 \text{ is indefinite}\} $ as in Theorem~\ref{thm1}.
\end{cor}
\begin{proof}
	[Proof of Corollary~\ref{derivat bound}.]

	Combined with the Lemma~\ref{holder lem} and the expression of $\gn_t^{\prime}(f)$ with $f= \exp(\gk N \cP(\ovR_{12}))$ in~\eqref{derivative}, and the symmetry between replicas, we can do the following analysis. When $\gD^2$ is non-negative definite, as in Remark~\ref{non-neg remark}, $\cP(\ovR_{12}) =\cQ(\ovR_{12})$. In this case, by just dropping the terms
	\[- 2\sum_{l  \le  2}\gn_t(\cQ(\ovR_{l,3})f) \]
	in~\eqref{derivative}, and applying the inequalities in Lemma~\ref{holder lem}, we get the upper bound of $\gn_t^{\prime}(f)$, which is $2N\gb^2\gn_t(\cP(\ovR_{12})\exp(\gk N \cP(\ovR_{12})))$. For the case $\gD^2$ is indefinite, similarly we just apply all the inequalities in Lemma~\ref{holder lem} to get the bound, \ie~$ 4N\gb^2\gn_t(\cP(\ovR_{12})\exp(\gk N \cP(\ovR_{12})))$.
\end{proof}
\begin{cor}\label{monotone}
	For $t <\gr/(2\ga \gb^2)$, the function
	\[ t \mapsto \gn_t(\exp((\gr-2\ga t\gb^2) N \cP(\ovR_{12}))) \]
	is non-increasing.
\end{cor}
\begin{proof}
	[Proof of Corollary~\ref{monotone}.] Taking the derivative of $\gn_t(\exp((\gr-2\ga t\gb^2) N \cP(\ovR_{12})))$ w.r.t.~$t$, we get
	\begin{align*}
		\frac{d}{dt} & (\gn_t(\exp((\gr-2\ga t\gb^2) N \cP(\ovR_{12}))))                                                                                           \\
		             & =\gn_t^\prime(\exp((\gr-2\ga t\gb^2) N \cP(\ovR_{12}))) - 2\ga N\gb^2\gn_t(\cP(\ovR_{12})\exp((\gr-2\ga t\gb^2) N \cP(\ovR_{12})))  \le  0.
	\end{align*}
	The last step is due to Corollary~\ref{derivat bound}.
\end{proof}

Our goal is to control $\gn(\exp(\gk N \cP(\ovR_{12})))$, by the idea of interpolation, we just need to obtain a bound for $\gn_0(\exp(\gk N \cP(\ovR_{12})))$ since the interpolation path is non-increasing by Corollary~\ref{monotone}.
\begin{lem}\label{concen of decouple}
	For $\gk<1/(2\norm{\cV})$, we have
	\begin{align}
		\gn_0(\exp(\gk N \cP(\ovR_{12}))) \le \det(I-2\gk \cV)^{-1/2}.
	\end{align}
\end{lem}
\begin{proof}
	[Proof of Lemma~\ref{concen of decouple}.] First, take a Gaussian vector $\mvu = \sqrt{2\gk} \gL^{\half} \cV^{\half} \mveta$, where $\mveta \sim \N(0,I_{m})$ and is independent with the disorder and Gibbs randomness, then
	\begin{align*}
		\gn_0(\exp(\gk N\cP(\ovR_{12})) =\E \gn_0(\exp(\sqrt{N}\mvu^{\T} \ovR_{12}))
		 & =\E\gn_0(\prod_{s=1}^m \exp(\sqrt{N}u_s(R^{s}-q_{s})))                                                               \\
		 & = \E\gn_0\left(\prod_{s=1}^m \prod_{i\in I_s} \exp\bigl(\sqrt{N}(\gs_{i}^1\gs_{i}^2-q_{s}){u_s}/{|I_s|}\bigr)\right) \\
		 & \le  \E \prod_{s=1}^m \exp(Nu_s^2/2|I_s|)
		= \E \exp(\mvu^{\T}\gL^{-1}\mvu/2),
	\end{align*}
	where the last inequality is by the fact that (see~\cite{Tal11a}*{Page 39} for a proof)
	\[\gn_0(\exp u(\gs_i^1\gs_i^2-q)) = \exp(-qu)(\cosh(u)+q\sinh(u)) \le \exp(u^2/2).\]
	Finally we use the fact that $\E \exp(\mvu^{\T}\gL^{-1}\mvu/2) = \E\exp(\gk \mveta^{\T} \cV \mveta) = \det(I-2\gk\cV)^{-\half}$
	and this completes the proof.
\end{proof}

Now we are ready to prove Theorem~\ref{thm1}.
\subsection{Proof of Theorem~\ref{thm1}}
Take $\gr = \gh +2\ga \gb^2<1/(2\norm{\cV})$, then
\begin{align*}
	\gn_t(\exp((\gh+2\ga(1-t)\gb^2) N \cP(\ovR_{12}))) =\gn_t(\exp((\gr-2\ga t\gb^2) N \cP(\ovR_{12})))
\end{align*}
By the non-increasing property of $\gn_t$ in Corollary~\ref{monotone}, we have
\[ \gn_t(\exp((\gr-2\ga t\gb^2) N \cP(\ovR_{12}))) \le \gn_0(\exp(\gr N \cP(\ovR_{12}))) \quad \text{for} \quad t \in [0,1]. \]
By Lemma~\ref{concen of decouple},
\[ \gn_0(\exp(\gr N \cP(\ovR_{12})))  \le  \det(I-2\gr \cV)^{-1/2} =\det(I-2(\gh +2\ga \gb^2) \cV)^{-1/2}. \]
Thus,
\[ \gn_1(\exp(\gh N \cP(\ovR_{12})))  \le  \det(I-2(\gh +2\ga \gb^2) \cV)^{-1/2}. \]
This proves Theorem~\ref{thm1}. \hfill\qed

\section{Asymptotic for the free energy}\label{clt-free-engy}
In this section, we give the proof of Theorem~\ref{RS solution}.

\subsection{Proof of Theorem~\ref{RS solution}}

First, let us prove the RS solution~\eqref{sol}. Recall the formula~\eqref{free-ener deriv} in Section~\ref{sec2: latala},
\[ \phi^{\prime}(t) = \frac{\gb^2}{4}\bigl(\cQ(\mathbf{1}-\vq) - \gn_t(\cQ(\ovR_{12}))\bigr) \]
which gives the derivative of free energy associated to $H_{N,t}(\mvgs)$. By Theorem~\ref{thm1}, when $\gb^2 < \gb_0^2$, we have $\gn_t(\cQ(\ovR_{12})) \le \frac{K}{N}$ by Jensen's inequality, which implies that $\abs{\phi'(t) - \frac{\gb^2}{4}\cQ(\bf1-\vq)} \le \frac{K}{N}$, combining with the fact that $\phi(1)- \phi(0) = \int_0^1 \phi'(t) dt$, then
\begin{align*}
	 & \abs{\phi(1)-\phi(0) - \frac{\gb^2}{4}\cQ(\bf1-\vq)} \le \frac{K}{N}
\end{align*}
where $\phi(1) = F_N(\gb,h)$, and $\phi(0)$ is the free energy with Hamiltonian $H_{N,0}(\mvgs)$, a simple calculation based on the definition of free energy can give us $\phi(0) = \log 2+ \sum_{r=1}^m \gl_r \E\log \cosh(\gb \eta \sqrt{Q^{r}}+h)$, where $Q^{r}=(\gD^{2}\gL\vq)_r$ and the expectation is w.r.t.~$\eta \sim \N(0,1)$.

Next, we prove the second part of Theorem~\ref{RS solution}, a quantitative Central Limit Theorem for the free energy with $h \neq 0$. Let
\[ X_t = \frac{1}{\sqrt{N}} \bigl(\log Z_{N,t}- \E \log Z_{N,0} - \frac14N \gb^2 t\cdot \cQ(\bf1-\vq)\bigr) \]
where $Z_{N,t}$ is the partition function associated with the interpolated Hamiltonian $H_{N,t}(\mvgs)$ introduced in Section~\ref{sec2: latala}, specifically

\[ Z_{N,t} = \sum_{\mvgs} \exp \biggl( \frac{\gb}{\sqrt{N}}\sum_{i<j}g_{ij}\gs_i \gs_j \cdot \sqrt{t} + \gb \sum_{r=1}^m \sqrt{Q^r} \sum_{i \in I_r} \eta_i \gs_i \cdot \sqrt{1-t} + h \sum_i \gs_i\biggr). \]

Our goal is to prove a Central Limit Theorem for $t=1$. Let us first introduce a lemma packing up technical details of the computation.
\begin{lem}\label{asy-lem1}
	Let $G_{t}(\mvgs):= \frac{\exp(H_{N,t}(\mvgs))}{Z_{N,t}}, \gs\in\{1,-1\}^{N}$ be the interpolated Gibbs measure. For any bounded twice differentiable function $g \in C_b^2(\bR)$ and $t \in [0,1]$, we have
	\begin{align}\label{asy-eq0}
		\frac{d}{dt} \E g(X_t) = \frac{\gb^2}{4} \E \left( g''(X_t)\bigl(\la \cQ(\ovR_{12})\ra_t - \cQ(\vq)-N^{-1}\tr(\gD^{2}\gL)\bigr) - \sqrt{N} g'(X_t)\la \cQ(\ovR_{12}) \ra_t\right).
	\end{align}
	In particular, for $m(t,\theta):= \E e^{\i \theta X_t}$ the characteristic function of $X_t$, we have
	\begin{align}\label{sec3:eq0}
		\abs{\frac{\partial}{\partial t}m(t,\theta) - v^2\theta^2 m(t,\theta) } \le \frac{\gb^2 C}{4}\left(\frac{|\theta|}{\sqrt{N}} + \frac{2\theta^2}{N}\right) \text{ for all }\theta\in\dR
	\end{align}
	where $C$ is an upper bound of $ \sup_{0\le t\le 1} N \nu_t(|\cQ(\ovR_{12})|)$ and $\tr(\gD^{2}\gL) $; and $v^2 :=\gb^2\cQ(\vq)/4$.
\end{lem}
\begin{proof}
	[Proof of Lemma~\ref{asy-lem1}.] For $g \in C_b^{2}(\bR)$, we evaluate the LHS directly,
	\begin{align*}
		 & \frac{d}{dt} \E g(X_t)                                                                                                                                                                                                                                            \\
		 & =\frac{1}{\sqrt{N}} \E g'(X_t) \biggl( \frac{\gb}{2 }\sum_{\mvgs} G_t(\mvgs) \biggl(\frac{1}{\sqrt{tN}}\sum_{i<j} g_{ij}\gs_i \gs_j - \frac{1}{\sqrt{1-t}} \sum_{r=1}^{m}\sqrt{Q^r} \sum_{i \in I_r}\eta_i \gs_i \biggr)- \frac{N\gb^2}{4} \cQ(\vone-\vq)\biggr).
	\end{align*}
	Notice that $\frac{1}{\sqrt{tN}}\sum_{i<j} g_{ij}\gs_i \gs_j - \frac{1}{\sqrt{1-t}} \sum_{r=1}^{m}\sqrt{Q^r} \sum_{i \in I_r}\eta_i \gs_i$ and $X_t$ are both Gaussian random variables. Applying Gaussian integration by parts, we get that the RHS is equal to
	\begin{align*}
		 & \frac{\gb^2}{2N} \E \Bigg[ g''(X_s)\sum_{\mvgs,\mvgt}G_t(\mvgs) G_t(\mvgt) \left( \frac12 (N\cQ(\vR_{12})-\sum_{r=1}^m \gl_r \gD_{r,r}^2) - N \sum_{r,s}\gl_rR_{12}^r\gD_{r,s}^2\gl_sq_s \right) \Bigg] \\
		 & \qquad\qquad + \frac{\gb^2}{2\sqrt{N}} \E \Bigg[ g'(X_s)\Bigg(\sum_{\mvgs} G_t(\mvgs) \Big( \frac12 (N\cQ(\vone)-\sum_{r=1}^m \gl_r \gD_{r,r}^2) - N \sum_{r,s}\gl_r\gD_{r,s}^2\gl_sq_s                 \\
		 & \qquad\qquad -\sum_{\mvgt}G_t (\mvgt)\Big(\frac12 (N\cQ(\vR_{12})-\sum_{r=1}^m \gl_r \gD_{r,r}^2)- N \sum_{r,s}\gl_rR_{12}^r\gD_{r,s}^2\gl_sq_s \Big)\Big)- \frac{N}{2}\cQ(\vone-\vq)\Bigg)\Bigg]       \\
		 & = \frac{\gb^2}{4} \E \left[ g''(X_t) \left( \la \cQ(\ovR_{12}) \ra_t - \cQ(\vq) - N^{-1}\tr(\gD^{2}\gL)\right) - \sqrt{N} g'(X_t)\la \cQ(\ovR_{12}) \ra_t \right].
	\end{align*}
	For the second part, taking $g(x)=e^{\i x\theta}$ for fixed $\theta\in\dR$, equation~\eqref{asy-eq0} can be rewritten as
	\begin{align*}
		\frac{\partial}{\partial t}m(t,\theta) & = \frac{\gb^2\theta^2}{4} \cQ(\vq) m(t,\theta)                                                                                                                                                \\
		                                       & \quad+ \frac{\gb^2}{4}\E \left(\Bigl( -\frac1N\theta^2\left( N\la \cQ(\ovR_{12}) \ra_t - \tr(\gD^{2}\gL)\right) - \i \sqrt{N} \theta\la \cQ(\ovR_{12}) \ra_t \Bigr) e^{\i \theta X_t}\right).
	\end{align*}
	From Section~\ref{sec2: latala}, we know that $N \nu_{t}|\cQ(\ovR_{12})|\le N \nu_{t}(\cP(\ovR_{12}))$ is bounded, then
	\begin{align*}
		\abs{\frac{\partial}{\partial t}m(t,\theta) - v^2\theta^2 m(t,\theta) } \le \frac{\gb^2 C}{4}\left(\frac{|\theta|}{\sqrt{N}} + \frac{2\theta^2}{N}\right) \text{ for all }t\in[0,1], \theta\in\dR
	\end{align*}
	where $\max \big\{ N\nu_t|\cQ(\ovR_{12})|, 0\le t\le 1, \tr(\gD^2\gL) \big\} \le C.$
\end{proof}

Now we prove the Central Limit Theorem for the free energy.
\begin{proof}
	[Proof of Theorem~\ref{RS solution} part~(ii)] Notice that
	\begin{align}\label{sec3:eq1}
		\abs{e^{-v^2\theta^2}m(1,\theta) - m(0,\theta)} & = \abs{\int_0^1 e^{-v^2\theta^2 t}\left( \partial_t m(t,\theta) - v^2\theta^2 m(t,\theta)\right) dt} \notag                                                                                           \\
		                                                & \le \frac{\gb^2 C}{4} \left(\frac{|\theta|}{\sqrt{N}} + \frac{2\theta^2}{N}\right) \int_0^1 e^{-v^2\theta^2 t} dt \le \frac{\gb^2 C}{4} \left(\frac{|\theta|}{\sqrt{N}} + \frac{2\theta^2}{N}\right).
	\end{align}
	Thus
	\begin{align*}
		\abs{m(1,\theta)- m(0,\theta)e^{v^2\theta^2}} \le \frac{\gb^2 C}{4} \left(\frac{|\theta|}{\sqrt{N}} + \frac{2\theta^2}{N}\right)e^{v^2\theta^2}.
	\end{align*}
	For $\theta$ fixed, let $N \to \infty$. The RHS of the last step goes to 0. Moreover, as $N\to\infty$, we have by the classical CLT $m(0,\theta) \to e^{-c^2\theta^2/2}$ where $$c^2=\lim_{N\to\infty}N^{-1}\var(\log Z_{N,0})=\sum_{s=1}^{m}\gl_{s}\var(\log\cosh(\gb\eta\sqrt{Q^{s}}+h)).$$
	This completes the proof.
\end{proof}


\section{Concentration of overlap with zero external field}\label{concen-h0}

In this section, we study the MSK model without an external field. We prove that when $\gb<\gb_c$, the MSK model is in the replica symmetric phase. Note that our argument holds when $\gD^2$ is indefinite. In that case, the RS regime is still given by $\gb<\gb_c$ as in $\gD^2$ \pd~case.

The proof of Theorem~\ref{ovp no h} is based on a control of the free energy. Here $q=0$ is the unique solution to~\eqref{syseq} and the averaged Gibbs measure under the decoupled Hamiltonian is the product of i.i.d.~Bernoulli($\half$) measures and thus is non-random. Before giving the proof, we introduce several useful lemmas following the proof for the classical case. Recall that $\cP(\vx)=\vx^{\T}\gL^{\half}\cV\gL^{\half}\vx$ and $\cV=|\gL^{\half}\gD^{2}\gL^{\half}|$.

\begin{lem}\label{use-lem}
	If $\gamma+\gb^2 < \gb_c^2$, then
	\begin{align}
		\E \sum_{\mvgs_1, \mvgs_2} \exp\left( H_N(\mvgs_1)+ H_N(\mvgs_2) + \frac{\gamma N}{2}\cP(\vR_{12})\right) \le (\E Z_N)^2\cdot \det(I-(\gb^2+\gamma)\cV)^{-\half}.
	\end{align}
\end{lem}
\begin{proof}
	[Proof of Lemma~\ref{use-lem}] Since the Hamiltonian $H_N{(\mvgs)}$ is a \emph{centered} Gaussian field, we know that
	\begin{align}\label{eq4-h0}
		\begin{split}
			\E \sum_{\mvgs_1, \mvgs_2} & \exp\left( H_N(\mvgs_1)+ H_N(\mvgs_2)+\frac{\gamma N}{2}\cP(\vR_{12})\right)\\
			& =\sum_{\mvgs_1, \mvgs_2}\exp\left(\frac{1}{2}\E(H_N(\mvgs_1)+ H_N(\mvgs_2))^2 + \frac{\gamma N}{2}\cP(\vR_{12})\right) \\
			& \le \sum_{\mvgs_1, \mvgs_2}\exp \left( \frac{N\beta^2}{2}\left(\sum_{s,t}\lambda_s \gD_{s,t}^2\lambda_t-\frac{2}{N}\sum_s\gl_s\gD_{s,s}^2\right)+\frac{(\gb^2+\gamma)N \cP(\vR_{12})}{2}\right) \\
			& \le (\E Z_N)^2 \cdot 2^{-2N} \sum_{\mvgs_1, \mvgs_2} \exp\left(\frac{\gd N}{2} \cP(\vR_{12})\right)
		\end{split}
	\end{align}
	where $\gd =\gb^2+\gamma$, and the last two inequalities are based on the fact that
	$$
		\E Z_N = 2^N\exp\left(\frac{N\beta^2}{4}\left(\sum_{s,t}\lambda_s \gD_{s,t}^2\lambda_t-\frac{1}{N}\sum_s\gl_s\gD_{s,s}^2\right)\right),
	$$
	and
	\begin{align*}
		\E H_N(\mvgs)^2             & = \frac{N\beta^2}{2}\left(\sum_{s,t}\lambda_s \gD_{s,t}^2\lambda_t-\frac{1}{N}\sum_s\gl_s\gD_{s,s}^2\right)       \\
		\E H_N(\mvgs_1)H_N(\mvgs_2) & =\frac{N\beta^2}{2}\left(\sum_{s,t}\lambda_s R^s\gD_{s,t}^2\lambda_tR^t-\frac{1}{N}\sum_s\gl_s\gD_{s,s}^2\right).
	\end{align*}
	Next we prove that, for fixed $\mvgs_2$,
	\begin{align}\label{sech0:eq1}
		2^{-N} \sum_{\mvgs_1} \exp \left( \frac{\gd N}{2} \cP(\vR_{12})\right) = 2^{-N} \sum_{\mvgs} \exp \left( \frac{\gd N}{2} \cP(\vv)\right)
	\end{align}
	where $\vv = (\frac{1}{|I_1|} \sum_{i \in I_1} \gs_i, \frac{1}{|I_2|} \sum_{i \in I_2} \gs_i, \ldots, \frac{1}{|I_m|} \sum_{i \in I_m} \gs_i)$. Note that for $\mvxi = (\xi_1, \xi_2, \ldots, \xi_m)$, where $\xi_s = \frac{1}{|I_s|} \sum_{i\in I_s}\eta_i$, and $\eta_i \sim \text{Bernouli}(\half)$ are i.i.d, we have
	\begin{align*}
		2^{-N} \sum_{\mvgs} \exp \left( \frac{\gd N}{2} \cP(\vv)\right) & =\E \exp \left(\frac{\gd N}{2}\cP(\mvxi) \right)                                                                   \\
		                                                                & = \E \exp\left(\frac{\gd N}{2} \cQ(\mvxi,\gL^{\half}\cV\gL^{\half}) \right) = \E\E_{\vu}\exp (\sqrt{N} \vu \mvxi )
	\end{align*}
	where $\vu \sim \N(0,\gd \gL^{\half} \cV\gL^{\half})$. By a similar way as in the proof of Lemma~\ref{concen of decouple}, we have
	\begin{align*}
		\E\E_{\vu}\exp (\sqrt{N} \vu \mvxi ) & = \E_{\vu} \E\prod_{s=1}^m \prod_{i \in I_s}\exp(\sqrt{N}u_s \eta_i/|I_s|)                                         \\
		                                     & \le \E_{\vu} \prod_{s=1}^m \exp(N u_s^2/2|I_s|) = \E_{\vu} \exp(\vu^{\T}\gL^{-1}\vu /2) = \det(I-\gd\cV)^{-\half}.
	\end{align*}
	Going back to~\eqref{eq4-h0} and using $\gd=\gb^{2}+\gc$, we finish the proof.
\end{proof}
\begin{thm}\label{lower devat}
	If $\gb < \gb_c$, then there exists a constant $K$ such that for each $N>0$, and each $t>0$,
	\begin{align}
		\pr\left( \log Z_N \le N(\log 2+ {\gb^2}/{4}\cdot \cQ(\vone)) -t\right) \le K \exp(-t^2/K).
	\end{align}
\end{thm}
The following lemma is going to tell us how to choose the $K$ appropriately.
\begin{lem}\label{lemm1-h0}
	If $\gb< \gb_c$, we have
	\begin{align}
		\pr \left(Z_N \ge \E Z_N/2, N\la \cQ(\vR_{12})\ra \le K \right)\ge \frac{1}{K}.
	\end{align}
\end{lem}
\begin{proof}
	[Proof of Lemma~\ref{lemm1-h0}] Notice that
	\[Z_N^2\la \exp(\frac{\gamma N}{2} \cQ(\vR_{12}))\ra = \sum_{\mvgs_1, \mvgs_2} \exp\left( H_N(\mvgs_1)+ H_N(\mvgs_2) + \frac{\gamma N}{2}\cQ(\vR_{12})\right).\]
	Thus by Lemma~\ref{use-lem}, we have
	\begin{align*}
		\E\left( Z_N^2\la \exp(\frac{\gamma N}{2} \cQ(\vR_{12}))\ra \right) \le (\E Z_N)^2 \det(I-(\gb^2+\gamma)\cV)^{-\half}
	\end{align*}
	and
	\[
		\pr(Z_{N}\ge \E Z_{N}/2) \ge (\E Z_N)^{2}/4\E Z_{N}^{2}\ge \det(I-\gb^2\cV)^{\half}/4.
	\]

	By Markov's inequality,
	\begin{align*}
		\pr \left( Z_N^2\la \exp(\frac{\gamma N}{2} \cQ(\vR_{12}))\ra \le t (\E Z_N)^2 \right) \ge 1- \frac1t \det(I-(\gb^2+\gamma)\cV)^{-\half}
	\end{align*}
	and
	\begin{align*}
		\pr \bigl( Z_N^2 & \la \exp(\frac{\gamma N}{2} \cQ(\vR_{12}))\ra \le t (\E Z_N)^2, Z_N \ge \E Z_N/2\bigr)                                            \\
		                 & \ge \pr \left( Z_N^2\la \exp(\frac{\gamma N}{2} \cQ(\vR_{12}))\ra \le t (\E Z_N)^2 \right)+\pr \left( Z_N \ge \E Z_N/{2}\right)-1 \\
		                 & \ge \pr \left( Z_N \ge \E Z_N/{2}\right)- \frac1t\det(I-(\gb^2+\gamma)\cV)^{-\half}.
	\end{align*}
	Based on this inequality, we can choose some large $K$ such that,
	\begin{align*}
		\pr \left( Z_N^2\la \exp(\frac{\gamma N}{2} \cQ(\vR_{12}))\ra \le K (\E Z_N)^2, Z_N \ge \E Z_N/2\right) \ge \frac{1}{K}.
	\end{align*}
	When $Z_N \ge \E Z_N/{2}$, $Z_N^2\la \exp(\frac{\gamma N}{2} \cQ(\vR_{12}))\ra \le K (\E Z_N)^2 $ implies that
	\[\left\la \exp(\frac{\gamma N}{2} \cQ(\vR_{12}))\right\ra \le 4K \]
	and hence $\la N\cQ(\vR_{12}))\ra \le K$. The constant $K$ in different expressions are different, here we abuse the notation a bit.
\end{proof}
\begin{lem}\label{lemm2-h0}
	The following identity:
	\begin{align}
		N^2 \la \cQ(\vR_{12}) \ra =\sum_{s,t}\sum_{i\in I_s} \sum_{j\in I_t}\gD_{s,t}^2 \la\gs_i \gs_j \ra^2\ge 0
	\end{align}
	holds.
\end{lem}
\begin{proof}
	[Proof of Lemma~\ref{lemm2-h0}] Recall that $\cQ(\vR_{12})$ is the quadratic form of $\vR_{12}$ with matrix $\gL\gD^2\gL$, then we have
	\begin{align*}
		N^2 \cQ(\vR_{12}) & = N^2\sum_{s,t} R_{12}^s \gl_s \gD_{s,t}^2 \gl_t R_{12}^t                                                                                                                                                                              \\
		                  & = N^2\sum_{s,t} \frac{1}{|I_s|} \sum_{i \in I_s} \gs_i^1 \gs_i^2 \gl_s \gD_{s,t}^2 \gl_t \frac{1}{|I_t|} \sum_{j \in I_t} \gs_j^1 \gs_j^2 = \sum_{s,t} \sum_{i \in I_s} \gs_i^1 \gs_i^2 \gD_{s,t}^2 \sum_{j \in I_t} \gs_j^1 \gs_j^2 ,
	\end{align*}
	After applying the Gibbs average $\la \cdot \ra$ on both sides, for RHS we have $\la \gs_i^1 \gs_j^1 \gD_{s,t}^2\gs_i^2 \gs_j^2\ra = \gD_{s,t}^2 \la \gs_i \gs_j \ra^2$, then $N^2\ \la \cQ(\vR_{12}) \ra = \sum_{s,t}\sum_{i\in I_s} \sum_{j\in I_t}\gD_{s,t}^2 \la\gs_i \gs_j \ra^2$.
\end{proof}
\begin{lem}
	[{\cite{Tal03}*{Lemma~2.2.11}}]\label{lemm-talagrand} Consider a closed subset $B$ of $\bR^M$, and set
	\[ d(\vx,B) = \inf\{d(\vx,\vy);\vy \in B\}, \]
	as the Euclidean distance from $\vx$ to $B$. Then for $t>0$, we have
	\begin{align}\label{eq0-h0}
		\pr\left(d(\mveta,B) \ge t+ 2 \sqrt{\log\frac{2}{\pr(\mveta\in B)}} \right) \le 2 \exp(-t^2/4)
	\end{align}
	where $\mveta = (\eta_1,\eta_2,\ldots, \eta_m)$ and $\eta_i \sim \N(0,1)$ are i.i.d.
\end{lem}
Now we are ready to prove the Theorem~\ref{lower devat}
\begin{proof}
	[Proof of Theorem~\ref{lower devat}] Recall the Hamiltonian with $h=0$ is
	\[ H_N(\mvgs) = \frac{\gb}{\sqrt{N}} \sum_{i<j} g_{ij}\gs_i\gs_j \]
	where $\E g_{ij}^2 = \gD^2_{s,t}$ for $i\in I_s, j\in I_t$. In the following content, we let $g_{ij} = \gD_{s,t} \eta_{ij}$, where $\eta_{ij} \sim \N(0,1)$. Let $M=N(N-1)/2$, and consider Gaussian $\mveta = (\eta_{ij})_{i<j}$, in this case, $\mveta \in \bR^M$. We understand $H_N(\mvgs), Z_N$ as functions of $\mveta$. By Lemma~\ref{lemm1-h0}, for some suitably large $K=K_1$, there is a subset $B \subset \bR^M$ with
	\[ \{\mveta \in B \} = \{ Z_N(\mveta) \ge \E Z_N/2; N\la \cQ(\vR_{12})\ra \le K_1 \} \]
	i.e. the set $B$ characterizes the event on RHS, and also $\pr(\mveta \in B) \ge \frac 1 {K_1}$. Next we will prove
	\begin{align}\label{eq1-h0}
		\log Z_N(\mveta) \ge N\left(\log 2 + \frac{\gb^2}{4}\cQ(\bf1)\right)- K_1(1+d(\mveta,B))
	\end{align}
	For $\mveta' \in B$, we have 
	\[
	\log Z_N(\mveta') \ge \log \frac 1 2 \E Z_N = \frac{\gb^2}{4}\left(N\cQ(\mathbf{1})- \sum_{s=1}^m\gl_s\gD_{s,s}^2\right)+(N-1)\log 2.
	\] To prove~\eqref{eq1-h0}, it is enough to show that
	\begin{align}\label{eq2-h0}
		\log Z_N(\mveta) \ge \log Z_N(\mveta') - K_1 d(\mveta,\mveta')
	\end{align}
	for all $\mveta'\in B$. Here $K_{1}$ must be bigger than $\log 2 +\frac{\beta^2}{4}\sum_{s=1}^m\gl_s\gD_{s,s}^2$. Notice that
	\begin{align*}
		Z_N(\mveta) = Z_N(\mveta') \left\la \exp \biggl( \frac{\gb}{\sqrt{N}} \sum_{s,t}\sum_{\substack{i \in I_s, j\in I_t, \\
				i<j}}\gD_{s,t}(\mveta_{ij}-\mveta_{ij}')\gs_i\gs_j \biggr) \right \ra'
	\end{align*}
	where $\la \cdot \ra'$ denotes the Gibbs average with disorders $\mveta'$. Since there is an exponential part, it's natural to apply Jensen's inequality,
	\[ \left\la \exp \biggl( \frac{\gb}{\sqrt{N}} \sum_{s,t}\sum_{\substack{i \in I_s, j\in I_t,\\
				i<j}}\gD_{s,t}(\mveta_{ij}-\mveta_{ij}')\gs_i\gs_j \biggr) \right \ra' \ge \exp \biggl( \frac{\gb}{\sqrt{N}} \sum_{s,t}\sum_{\substack{i \in I_s, j\in I_t,\\
				i<j}}\gD_{s,t}(\mveta_{ij}-\mveta_{ij}')\la \gs_i\gs_j \ra' \biggr) \]
	where we use $g_{ij}=\gD_{s,t}\eta_{ij}$ to rewrite the Hamiltonian, then applying Cauchy-Schwarz inequality and by Lemma~\ref{lemm2-h0}, we have
	\[ \sum_{s,t}\sum_{\substack{i \in I_s, j\in I_t,\\
				i<j}}\gD_{s,t}(\mveta_{ij}-\mveta_{ij}')\la \gs_i\gs_j \ra' \ge -d(\mveta,\mveta')(N^2\la\cQ(\vR_{12})\ra')^{\half} \ge -K_1\sqrt{N}d(\mveta,\mveta') \]
	where $d(\mveta,\mveta') = \sum_{i<j}(\eta_{ij}-\eta_{ij}')^2$. The last inequality is based on $\mveta' \in B$, then this proves~\eqref{eq2-h0}, and hence~\eqref{eq1-h0}. From that, we have
	\[ \log Z_N(\mveta) \le N\left(\frac{\gb^2}{4}\cQ(\mathbf{1}) + \log 2\right) -t \Rightarrow d(\mveta,B) \ge \frac{t-K_1}{K_1} \ge 2 \sqrt{\log \frac{2}{\pr(\mveta \in B)}} +\frac{t-K_2}{K_1} \]
	then by~\eqref{eq0-h0} in Lemma~\ref{lemm-talagrand}, it follows
	\begin{align*}
		\pr \left( \log Z_N(\mveta) \le N \left( \frac{\gb^2}{4}\cQ(\mathbf{1}) +\log 2\right)-t\right) \le 2 \exp \left( -\frac{(t-K_2)^2}{4K_1^2} \right)
	\end{align*}
	for $t \ge K_2$. Then for the RHS, if we take $K$ large enough, when $t\ge K_2$,
	\[
		2 \exp \left( -\frac{(t-K_2)^2}{4K_1^2} \right) \le K \exp \left(-\frac{t^2}{K} \right).
	\]
	On the other hand, when $0 \le t \le K_2$, we have $ K \exp \left(-\frac{t^2}{K} \right) \ge 1$. Therefore, in any case, we proved
	\[ \pr \left( \log Z_N(\mveta) \le N \left( \frac{\gb^2}{4}\cQ(\mathbf{1}) +\log 2\right)-t\right) \le K \exp \left(-\frac{t^2}{K} \right). \]
\end{proof}

Now we collect all the above to prove the main theorem.
\subsection{Proof of Theorem~\ref{ovp no h}}
By Cauchy-schwarz w.r.t the measure $\la \cdot \ra$
\begin{align*}
	\left \la \exp\left(\frac{\gb_c^2-\gb^2}{8}N \cP(\vR_{12})\right) \right \ra & \le \left\la \exp\left(\frac{\gb_c^2-\gb^2}{4}N \cP(\vR_{12})\right) \right \ra^{\half}                                \\
	                                                                             & = \frac{1}{Z_N}\left(Z_N^2 \left \la\exp\left(\frac{\gb_c^2-\gb^2}{4}N \cP(\vR_{12})\right) \right \ra \right)^{\half}
\end{align*}
then apply Cauchy-schwarz again for $\E$,
\begin{align*}
	\E \left \la \exp\left(\frac{\gb_c^2-\gb^2}{8}N \cP(\vR_{12})\right)\right \ra & \le \left(\E \frac{1}{Z_N^2} \right)^{\half} \left[\E \left( Z_N^2 \left \la \exp\left(\frac{\gb_c^2-\gb^2}{4}N \cP(\vR_{12})\right)\right \ra \right)\right]^{\half} \\
	                                                                               & \le \left(\E \frac{1}{Z_N^2} \right)^{\half} (K\E(Z_N)^2)^{\half}
	= K \left(\E \frac{(\E Z_N)^2}{Z_N^2}\right)^{\half}
\end{align*}
where the second inequality is due to Lemma~\ref{use-lem}. By Theorem~\ref{lower devat},
\begin{align}
	\pr \left(\frac{\E Z_N}{Z_N} >t\right) \le K \exp(-(\log t)^2/K)
\end{align}
then we have $\E \frac{(\E Z_N)^2}{Z_N^2} <K$ by standard tail estimates. \hfill\qed

\section{MSK cavity solution}\label{sec:cavity}
For the classical SK model, the cavity method is an induction on the number of spins. In this section, we generalize the cavity method to the MSK model and derive a linear system for the variance-covariance matrices of the overlap vectors. More specifically, to compare the system with $N$ spins and $(N-1)$ spins, we need to choose a spin to be decoupled from others.

Since the classical SK model is the one-species case of the MSK model, the spin can be chosen uniformly. For the multi-species case, decoupling a spin now depends on the structure of the species. However, we can choose a species $s \in \{1,2,\ldots,m\}$ first with probability ${\gl_s}$, then select a spin uniformly inside that species to be decoupled, then compare this decoupled system and the original one. For convenience, we denote the released spin as $\eps := \gs_{s_{\star}}$, where $s_{\star}$ is the index of last spin in species $s$ for a configuration $\mvgs$. Once the species is determined, we take the convention to decouple the last spin in the chosen species because of symmetry inside a particular species. Note that this is equivalent to choosing a species, u.a.r.~from all spins.

Let $\gd$ represent the species to be chosen, which is a random variable taking value in $\{1,2,\ldots,m\}$ with probability $\mathbb{P}(\gd = s) = \gl_s$. Given $\gd=\tilde s$, we decouple the last spin $\gs_{\tilde s_{\star}}$ in that species, then the interpolated Hamiltonian between the original system and the decoupled system is
\begin{align}\label{cavity-H}
	H_t^{\tilde s}(\mvgs) = H_{\tilde s_{\star}^-}(\mvrho)+ \gs_{\tilde s_{\star}} \biggl( \sqrt{t} \cdot \frac{\gb}{\sqrt{N}}\sum_{i \neq \tilde s_{\star}}g_{i \tilde s_{\star}}\gs_i + \sqrt{1-t}\cdot \gb \eta \sqrt{Q^{\tilde s}} + h\biggr)
\end{align}
where $H_{\tilde s_{\star}^-}(\mvrho)= \frac{1}{\sqrt{N}}\sum_{i<j}g_{ij}\rho_i\rho_j + \sum_{i\neq \tilde s_{\star}}\rho_i$, and $\mvrho$ is a ($N-1$)-dimensional vector of spins in $\mvgs$ without $\gs_{\tilde s_{\star}}$.

Recall that the overlap vector between two replicas $\gs^l,\gs^{l^{\prime}}$ in MSK model is defined as
\begin{align}
	\vR_{ll^{\prime}} = (R_{ll^{\prime}}^{(1)}, R_{ll^{\prime}}^{(2)}, \cdots, R_{ll^{\prime}}^{(m)})^{\T}
\end{align}
where each coordinate in the above vector represents the marginal overlap in the corresponding species. Similarly,
\begin{align}
	\vR_{ll^{\prime}}^{(\tilde s-)} = (R_{ll^{\prime}}^{(1)}, R_{ll^{\prime}}^{(2)}, \cdots, R_{ll^{\prime}}^{(\tilde s-)},\cdots R_{ll^{\prime}}^{(m)})^{\T}
\end{align}
where $R_{ll^{\prime}}^{(\tilde s-)} = R_{ll^{\prime}}^{(\tilde s)} - \frac{1}{|I_\gd|}\eps_l\eps_{l^{\prime}}$, and $\eps_l=\mvgs^l_{\tilde s_{\star}}$ denotes the released spin in $\mvgs^l$.

In the following, $\gn_{t,\tilde s}(\cdot)$ denotes the expectation over the Gibbs measure and disorder $g_{ij}$ associated with the Hamiltonian~\eqref{cavity-H}. In particular, $\gn_{1,\tilde s}(\cdot)$ corresponds to the original system and does not depend on $\tilde s$, but $\gn_{0,\tilde s}$ and $\gn_{t,\tilde s}$ depend on $\tilde s$ since they both involve the decoupling procedure. In the rest of this section, we state the results for each fixed $\gd=\tilde s$, unless we state clearly otherwise.
\begin{lem}\label{prod property}
	For any $f^{-}$ on $\gS_{N-1}^n$, where $\gS_{N-1}$ does not contain the spin $\gs_{\tilde{s}_{*}}$, and $I \subset \{1,2,\cdots,n\}$, we have
	\begin{align}
		\gn_{0,\tilde s}\biggl(f^{-} \prod_{i\in I} \eps_i\biggr) = \gn_{0,\tilde s}(f^{-})\cdot \E(\tanh Y^{\tilde s})^{|I|}
	\end{align}
	where $Y^{\tilde s} = \gb \eta \sqrt{Q^{\tilde s}}+h$, and $|I|$ is the cardinality of the set $I$.
\end{lem}

The proof of this lemma is same to the classical case, see Section 1.4 in~\cite{Tal11a}. Next we turn to computation of $\gn_{t,\tilde s}'(f):=\frac{d}{dt}\gn_{t,\tilde s}(f)$ in cavity method. Let
\begin{align*}
	u_{\mvgs} = \frac{\gb}{\sqrt{N}} \gs_{\tilde s_{\star}} \sum_{i \neq \tilde s_{\star}} g_{i\tilde s_{\star}} \gs_i,\qquad v_{\mvgs} = \gb \eta \gs_{\tilde s_{\star}}\sqrt{Q^{\tilde s}},\qquad w_{\mvgs} = \exp(H_{\tilde s_{\star}^-}(\mvgr)+ h\gs_{\tilde s_{\star}})
\end{align*}
in~\eqref{cavity-H}. Recall that in Section~\ref{sec2: latala}, the derivative $\gn_t'(f)$ was computed in (\ref{eq:gen-derivative}) for some $f$ on $\gS_N^n$:
\begin{align*}
	\gn_t'(f) = 2\biggl(\sum_{1 \le  l < l^{\prime}  \le  n} \gn_t(U(\mvgs^l,\mvgs^{l^{\prime}})f) - n \sum_{l  \le  n}\gn_t(U(\mvgs^l,\mvgs^{n+1})f) + \frac{n(n+1)}{2} \gn_t(U(\mvgs^{n+1},\mvgs^{n+2})f)\biggr)
\end{align*}
where $U(\mvgs^l, \mvgs^{l^{\prime}}) = \frac{1}{2}(\E u_{\mvgs^l}u_{\mvgs^{l^{\prime}}} -\E v_{\mvgs^l}v_{\mvgs^{l^{\prime}}}) $. In the setting of cavity method,
\begin{align*}
	\E u_{\mvgs^l}u_{\mvgs^{l^{\prime}}} & = \gb^2 \eps_l \eps_{l^{\prime}}\biggl(\sum_{s \neq \tilde s} \gD_{s,\tilde s}^2\gl_s R_{ll^{\prime}}^{(s)} + \gD_{\tilde s,\tilde s}^2 \gl_{\tilde s}R_{ll^{\prime}}^{(\tilde s-)}\biggr) \\
	\E v_{\mvgs^l}v_{\mvgs^{l^{\prime}}} & = \gb^2 \eps_l \eps_{l^{\prime}} Q^{\tilde s}
\end{align*}
then
\begin{align*}
	U(\mvgs^l, \mvgs^{l^{\prime}}) & = \frac{1}{2}\gb^2 \eps_l \eps_{l^{\prime}}\biggl(\sum_{s \neq \tilde s} \gD_{s, \tilde s}^2\gl_s R_{ll^{\prime}}^{(s)} + \gD_{\tilde s,\tilde s}^2\gl_{\tilde s}R_{ll^{\prime}}^{(\tilde s-)} - Q^{\tilde s}\biggr) = \frac{1}{2}\gb^2 \eps_l \eps_{l^{\prime}}\cdot \gD_{\tilde s}^2 \gL \ovR_{\ell\ell'}^{(\tilde s-)}.
\end{align*}
We make the inner product of vectors and the dependence of $U(\mvgs^l, \mvgs^{l^{\prime}})$ on $\tilde s$ implicit in the above expression. We use $\gD_{s}^2:= \gD^2_{s,\cdot}$ to denote the $s$-th row vector of $\gD^2$, and we keep using this notation for other symmetric matrices $A$ in the rest of this section. Thus in cavity solution $\gn'_{t,\tilde s}(f)$ can be written in the following way:
\begin{thm}\label{cav_deriv}
	For $f$ on $\gS_N^n$, we have
	\begin{align}\label{der formula}
		\begin{split}
			\gn'_{t,\tilde s}(f) & = \gb^2\biggl( \sum_{1 \le  l < l^{\prime}  \le  n} \gn_{t,\tilde s}(\eps_l \eps_{l^{\prime}}f\cdot \gD_{\tilde s}^2 \gL \ovR_{\ell\ell'}^{(\tilde s-)}) - n \sum_{l  \le  n}\gn_{t,\tilde s}(\eps_l \eps_{n+1}f\cdot \gD_{\tilde s}^2 \gL \ovR_{l,n+1}^{(\tilde s-)})\\
			&\qquad\qquad\qquad + \frac{n(n+1)}{2} \gn_{t,\tilde s}(\eps_{n+1} \eps_{n+2}f\cdot \gD_{\tilde s}^2 \gL \ovR_{n+1,n+2}^{(\tilde s-)}) \biggr).
		\end{split}
	\end{align}
\end{thm}

The next proposition is about some H\"{o}lder type inequalities.
\begin{prop}\label{Holder prop}
	For $f$ on $\gS_N^n$, and $\gt_1>0, \gt_2>0$, $1/\gt_1+1/\gt_2 = 1$, we have
	\begin{align}
		\abs{\gn(f)- \gn_{0,\tilde s}(f)}                       & \le  K(n,\gb,\tilde s) \gn(\abs{f}^{\gt_1})^{1/\gt_1} \gn\left(\abs{\gD_{\tilde s}^2 \gL (\vR^{(\tilde s-)}_{12}-\vq)}^{\gt_2}\right)^{1/\gt_2} \label{h1} \\
		\abs{\gn(f)- \gn'_{0,\tilde s}(f)- \gn_{0,\tilde s}(f)} & \le  K(n,\gb,\tilde s) \gn(\abs{f}^{\gt_1})^{1/\gt_1} \gn\left(\abs{\gD_{\tilde s}^2 \gL (\vR^{(\tilde s-)}_{12}-\vq)}^{2\gt_2}\right)^{1/\gt_2} .
		\label{h2}\end{align}
\end{prop}
\begin{rem}
	The proof of this theorem is similar to the original proof in Talagrand's book. For the inequality~\eqref{h1}, we just need to control $\gn'_{t,\tilde s}(f)$, and by Theorem~\ref{cav_deriv}, it suffices to control the terms $\gn_{t,\tilde s}(\eps_l \eps_{l^{\prime}}f(\gD_{\tilde s}^2 \gL \ovR_{\ell\ell'}^{(\tilde s-)})) $. We claim that $\abs{\gn_{t,\tilde s}(f)} \le \exp(4C_{\tilde s}n^2\gb^2)\gn(f)$, where $C_{\tilde s} :=\max\{\gD_{\tilde s}^2\gL\}$ is the maximal entry of the vector $\gD_{\tilde s}^2\gL$. Because in Theorem~\ref{cav_deriv}, $\abs{R_{12}^s-q_s}\le 2$ implies
	\[ \abs{\gn'_{t,\tilde s}(f)} \le 4 C_{\tilde s}n^2\gb^2\gn_{t,\tilde s}(f) \]
	then integrating the above inequality will prove the claim. Finally applying H\"older inequality, it proves~\eqref{h1}. For~\eqref{h2}, we just control the second order derivative $\gn''_{t,\tilde s}$ in a similar way.
\end{rem}

Recall the notation $\ov{\vR}_{12} = \vR_{12} -\vq$, using the Proposition~\ref{Holder prop}, one can estimate the quantities $\gn(\ov \vR_{12}^{\T} A\ov \vR_{12})$, $\gn(\ov \vR_{13}^{\T} A\ov \vR_{12})$ and $\gn(\ov \vR_{34}^{\T} A\ov \vR_{12})$, where $A=(\!(a_{s,t})\!)_{s,t=1}^m$ is a symmetric $m \times m$ matrix.

We start with $\gn(\ov \vR_{12}^{\T}A \ov \vR_{12})$, then use the cavity method to do some second moment computation to derive a linear system.
We incorporate the randomness of $\gd$ in the following to get
\begin{align}\label{eq:s3a}
	\begin{split}
		\gn\big(\ov{\vR}^{\T}_{12} \gL A \gL \ov{\vR}_{12}\big) &= \gn\big(\sum_{s=1}\gl_s \oR_{12}^{(s)}\sum_{t=1}\gl_ta_{s,t}\oR_{12}^{(t)}\big) \\
		&= \gn\big(\E^{\gd}\big( \oR_{12}^{(\gd)}\sum_{t=1}\gl_t a_{\gd,t} \oR_{12}^{(t)}\big)\big) = \hat{\gn}\big(\oR_{12}^{(\gd)}\sum_{t=1}\gl_t a_{\gd,t} \oR_{12}^{(t)}\big)
	\end{split}
\end{align}
where $\hat{\gn} := \gn \otimes \E^{\gd}$, and
\[ \oR_{12}^{(\gd)} = \frac{1}{|I_{\gd}|} \sum_{i \in I_\gd}(\gs_i^1\gs_i^2-q_{\gd}) = \frac{1}{|I_\gd|} \sum_{\substack{i \in I_\gd, \\
			i \neq \gd_{\star}}} \gs_i^1 \gs_i^2 + \frac{1}{|I_\gd|}\eps_1\eps_2 - q_{\gd}. \]
By using the symmetry among sites inside the species $\gd$, we continue the above expression~\eqref{eq:s3a}:
\begin{align}\label{eq:s3b}
	\begin{split}
		\hat{\gn}\biggl((\eps_1\eps_2-q_{\gd})\sum_{t=1}\gl_t a_{\gd,t} \oR_{12}^{(t)}\biggr) &= \hat{\gn}\biggl((\eps_1\eps_2-q_{\gd})\bigl(\sum_{t \neq \gd} \gl_t a_{\gd,t} \oR_{12}^{(t)}+\gl_{\gd}a_{\gd,\gd} \oR_{12}^{(\gd)}\bigr)\bigg)\\
		& = \hat{\gn}((\eps_1\eps_2-q_{\gd})A_{\gd}\gL\ov{\vR}_{12}^{(\gd-)}) + \frac1N\hat{\gn}((\eps_1\eps_2-q_{\gd}) a_{\gd,\gd}\eps_1\eps_2).
	\end{split}
\end{align}
Note that, in the above expression~\eqref{eq:s3b}, there is randomness inside $\eps_1, \eps_2, \gd$, which describes the way of choosing the species to decouple the spin, and the computation proceeds in a quenched way, because all the randomness only comes from $\gd$, which takes finitely many values. In the following subsections, we deal with the two terms in~\eqref{eq:s3b} separately.

\subsection{Estimation of $\hat{\gn}((\eps_1\eps_2-q_{\gd})A_{\gd}\gL\ov{\vR}_{12}^{(\gd-)})$ }

Recall the definition of $\hat{\gn}$ below~\eqref{eq:s3a} and notice that
\begin{align}\label{eq:s3bb}
	\hat{\gn}((\eps_1\eps_2-q_{\gd})A_{\gd}\gL\ov{\vR}_{12}^{(\gd-)}) = \sum_{s=1}^m \gl_{s}\cdot \gn((\eps_1\eps_2-q_{s})A_{s}\gL\ov{\vR}_{12}^{(s-)}).
\end{align}
Therefore it suffices to compute $ \gn((\eps_1\eps_2-q_{s})A_{s}\gL\ov{\vR}_{12}^{(s-)})$. First let's take
\[\tilde{f} =(\eps_1\eps_2-q_{s})A_{s}\gL\ov{\vR}^{(s-)}_{12},\]
by the second H\"older inequality~\eqref{h2} with $\gt_1 = 3, \gt_2 = 3/2$ in the Proposition~\ref{Holder prop}, we have
\[ \gn(\tilde{f}) =\gn_{0,s}'(\tilde{f})+ \mathfrak{R}. \]
Because $\gn_{0,s}(\tilde{f}) = 0 $ by Lemma~\ref{prod property}, where $\mathfrak{R}$ represents the higher order terms with
\begin{align}\label{error}
	\fR\le K\left(N^{-3/2} + \sum_{s=1}^{m}\nu(|R_{12}^{(s-)}-q_s|^{3})\right)\max_{s,t}|A_{s,t}|.
\end{align}
For the derivative term, by Theorem~\ref{cav_deriv}, $\gn_{0,s}'(\tilde{f})$ is a sum of terms in the form of
\[ \gb^2\gn_{0,s}(\eps_l\eps_{l^{\prime}}(\eps_1\eps_2-q_{s})\cdot A_{s}\gL\ovR_{12}^-\cdot \gD_{s}^2 \gL \ovR_{\ell\ell'}^{(s-)}). \]
In particular, we introduce a general formula as a corollary to Theorem~\ref{cav_deriv} to compute $\gn_{0,s}'(\tilde{f})$.
\begin{cor}\label{cor:4}
	Consider a function $f^-$ on $\gS_{N-1}^n$ and two integers $x \neq y  \le  n$. Then
	\begin{align}\label{cavity:1}
		\begin{split}
			\gn_{0,s}'\left((\eps_x\eps_y-q_{s})f^-\right) & = \sum_{1 \le  l < l^{\prime}  \le  n} b_{s}(l,l^{\prime};x,y) \gn_{0,s}(f^-\cdot \gD_{s}^2 \gL \ovR_{\ell\ell'}^{(s-)}) \\
			&\qquad\qquad - n \sum_{l  \le  n}b_{s}(l,n+1;x,y)\gn_{0,s}(f^-\cdot \gD_{s}^2 \gL \ovR_{l,n+1}^{(s-)})\\
			&\qquad\qquad\qquad\qquad + \frac{n(n+1)}{2} b_{s}(0)\gn_{0,s}(f^-\cdot \gD_{s}^2 \gL \ovR_{n+1,n+2}^{(s-)})
		\end{split}
	\end{align}
	where $b_{s}(l,l^{\prime};x,y) = b_{s}(\abs{\{l,l^{\prime}\} \cap \{x,y\}})$ and
	\begin{align}\label{def of b}
		\begin{split}
			b_{s}(2) & = \gb^2 \gn_{0,s}(\eps_1\eps_2(\eps_1\eps_2-q_{s})) =\gb^2 \gn_{0,s}(1-\eps_1\eps_2q_s) = \gb^2(1-q_s^2) \\
			b_{s}(1) & = \gb^2 \gn_{0,s}(\eps_1\eps_3(\eps_1\eps_2-q_{s})) =\gb^2 \gn_{0,s}(\eps_2\eps_3-\eps_1\eps_3q_{s}) =\gb^2 q_s(1-q_s) \\
			b_{s}(0) & = \gb^2 \gn_{0,s}(\eps_3\eps_4(\eps_1\eps_2-q_{s})) =\gb^2 \gn_{0,s}(\eps_1\eps_2\eps_3\eps_4-\eps_3\eps_4q_s) =\gb^2 (\hat{q}_{s}-q_s^2)
		\end{split}
	\end{align}
\end{cor}
In the above Corollary, $$q_{s}=\E\tanh^2(\gb \eta\sqrt{(\gD^{2}\gL\vq)_s}+h), \qquad \hat{q}_{s} :=\E\tanh^{4}(\gb \eta \sqrt{(\gD^{2}\gL\vq)_s}+h), $$ and we used Lemma~\ref{prod property} to get
\begin{align*}
	 & \gn_{0,s}(\eps_l\eps_{l^{\prime}}(\eps_1\eps_2-q_{s})\cdot A_{s}\gL\ovR_{12}^{(s-)}\cdot \gD_{s}^2 \gL \ovR_{\ell\ell'}^{(s-)})                      \\
	 & \qquad = \gn_{0,s}(\eps_l\eps_{l^{\prime}}(\eps_1\eps_2-q_{s}))\cdot \gn_{0,s}(A_{s}\gL\ovR_{12}^{(s-)}\cdot \gD_{s}^2 \gL \ovR_{\ell\ell'}^{(s-)}).
\end{align*}
Using Corollary~\ref{cor:4} with $n=2,x=1,y=2, f^- = A_{s}\gL\ovR_{12}^{(s-)}$, we have
\begin{align}\label{deriva 0}
	\begin{split}
		&\gn'_{0,s}((\eps_x\eps_y-q_{s})f^-)\\ & = \gn_{0,s}'\bigl((\eps_1\eps_2-q_{s})\cdot A_{s}\gL\ovR_{12}^{(s-)}\bigr)\\
		&= b_{s}(2)\cdot \gn_{0,s}(A_{s}\gL\ovR_{12}^{(s-)}\cdot \gD_{s}^2 \gL \ovR_{12}^{(s-)}) - 2 b_{s}(1)\cdot \gn_{0,s}(A_{s}\gL\ovR_{12}^{(s-)}\cdot \gD_{s}^2 \gL \ovR_{13}^{(s-)})\\
		&\qquad - 2 b_{s}(1)\cdot \gn_{0,s}(A_{s}\gL\ovR_{12}^{(s-)}\cdot \gD_{s}^2 \gL \ovR_{23}^{(s-)})
		+ 3 b_{s}(0)\cdot \gn_{0,s}(A_{s}\gL\ovR_{12}^{(s-)}\cdot \gD_{s}^2 \gL \ovR_{34}^{(s-)}).
	\end{split}
\end{align}
For terms like $\gn_{0,s}(\gL A_{s}\ovR_{12}^{(s-)}\cdot \gD_{s}^2\gL \ovR_{\ell\ell'}^{(s-)})$ in the above expression, we apply H\"older inequality~\eqref{h1} with $\gt_1 = \frac 3 2, \gt_2= 3$, to obtain
\[ \gn_{0,s}(\gL A_{s}\ovR_{12}^{(s-)}\cdot \gD_{s}^2\gL \ovR_{\ell\ell'}^{(s-)}) = \gn(\gL A_{s}\ovR_{12}^{(s-)}\cdot \gD_{s}^2\gL \ovR_{\ell\ell'}^{(s-)}) + \fR \]
and changing from $ \gn(\gL A_{s}\ovR_{12}^{(s-)}\cdot \gD_{s}^2\gL \ovR_{\ell\ell'}^{(s-)})$ to $ \gn(\gL A_{s}\ovR_{12}\cdot \gD_{s}^2\gL \ovR_{\ell\ell'})$ will result in an error term $\fR$, \ie
\[ \gn(\gL A_{s}\ovR_{12}^{(s-)}\cdot \gD_{s}^2\gL \ovR_{\ell\ell'}^{(s-)}) = \gn(\gL A_{s}\ovR_{12}\cdot \gD_{s}^2\gL \ovR_{\ell\ell'})+\fR. \]
This is because $R_{ll^{\prime}}^{(s-)} = R_{ll^{\prime}}^{(s)} - \frac{1}{|I_s|}\eps_l\eps_{l^{\prime}}$, and the fact that $\frac 1 N \gn(\gL A_{s}\ovR_{12}) = \fR$ as $xy \le x^{\frac 3 2}+y^3$ for $x,y \ge 0$. Combining~\eqref{deriva 0} and~\eqref{eq:s3bb}, we have
\begin{align}
	\begin{split}
		\hat{\gn}\big((\eps_1\eps_2-q_{\gd})A_{\gd}\gL\ov{\vR}_{12}^{(\gd-)}\big) &= \gn\big(\ovR_{12}^{\T}S_{A} B(2)\gD^2\gL\ovR_{12}\big) - 4\gn\big(\ovR_{12}^{\T}S_{A} B(1)\gD^2\gL\ovR_{13}\big)\\
		&\qquad + 3\gn\big(\ovR_{12}^{\T}S_{A} B(0)\gD^2\gL\ovR_{34}\big) + \fR
	\end{split}
\end{align}
where $S_{A}:=\gL A\gL$,
\begin{align}\label{def of B}
	B(i) = \diag(b_1(i), b_2(i),\ldots, b_m(i)), \quad \text{for} \quad i=0,1,2,
\end{align}
with $b_{s}(i)$ as given in~\eqref{def of b} and $\fR$ represent the higher order terms as defined in~\eqref{error}.

\subsection{Estimation of $\hat{\gn}((\eps_1\eps_2-q_{\gd})a_{\gd,\gd}\eps_1\eps_2)$}

For the second term in the last step of~\eqref{eq:s3b}, we get
\begin{align}\label{eq:s3c}
	\begin{split}
		\hat{\gn}((\eps_1\eps_2-q_{\gd})a_{\gd,\gd}\eps_1\eps_2) & = \hat{\gn}_{0,\gd}(a_{\gd,\gd}(1-\eps_1\eps_2q_{\gd})) +\fR\\
		& = \E^{\gd}(a_{\gd,\gd}(1-q_{\gd}^2))+\fR = \sum_{s=1}^m \gl_s a_{s,s}(1-q_s^2)+\fR.
	\end{split}
\end{align}
In the above computation, the first equality is due to the H\"older inequality~\eqref{h1} with $\gt_1=\infty, \gt_2 = 1$ in the Proposition~\ref{Holder prop}, the second equality is due to the property of $\gn_{0,s}$ as in Lemma~\ref{prod property}.

By collecting all the terms, we have
\begin{align}\label{U2}
	\begin{split}
		\gn(\ov{\vR}^{\T}_{12} S_{A} \ov{\vR}_{12}) & = \frac1N \sum_{s=1}^m \gl_s a_{s,s}(1-q_s^2)+ \gn\big(\ovR_{12}^{\T}S_{A} B(2)\gD^2\gL)\ovR_{12}\big) \\
		&\qquad - 4\gn\big(\ovR_{12}^{\T}S_{A} B(1)\gD^2\gL\ovR_{13}\big) + 3\gn\big(\ovR_{12}^{\T}S_{A} B(0)\gD^2\gL\ovR_{34}\big) + \fR.
	\end{split}
\end{align}

By a similar argument, we get
\begin{align}\label{U1}
	\begin{split}
		\gn(\ov{\vR}^{\T}_{12} S_{A} \ov{\vR}_{13}) & = \frac1N\sum_{s=1}^m \gl_s a_{s,s}(q_s-q_s^2)+\gn(\ovR_{12}^{\T}S_{A} B(1)\gD^2\gL\ovR_{12}) \\
		&\qquad + \gn(\ovR_{12}^{\T}S_{A} (B(2)-2B(1)-3B(0))\gD^2\gL\ovR_{13})\\
		&\qquad\qquad + \gn(\ovR_{12}^{\T}S_{A} (6B(0)-3B(1))\gD^2\gL\ovR_{34}) + \fR,
	\end{split}
\end{align}
and
\begin{align}\label{U0}
	\begin{split}
		\gn(\ov{\vR}^{\T}_{12} S_{A} \ov{\vR}_{34}) & = \frac1N\sum_{s=1}^m \gl_s a_{s,s}(\hat{q}_s-q_s^2)+\gn(\ovR_{12}^{\T}S_{A} B(0)\gD^2\gL\ovR_{12}) \\
		&\qquad + \gn(\ovR_{12}^{\T}S_{A} (4B(1)-8B(0))\gD^2\gL\ovR_{13})\\
		&\qquad\qquad + \gn(\ovR_{12}^{\T}S_{A} (B(2)-8B(1)+10B(0))\gD^2\gL\ovR_{34}) + \fR.
	\end{split}
\end{align}

We define the $m\times m$ diagonal matrices
\begin{align}\label{Theta}
	\Theta(i) :=\gb^{-2}B(i)\gL^{-1},\qquad \text{ for } i=0,1,2,
\end{align}
so that
\begin{align*}
	\sum_{s=1}^m \gl_s a_{s,s}(1-q_s^2) = \tr(S_{A}\Theta(2)),\quad & \quad \sum_{s=1}^m \gl_s a_{s,s}(q_s-q_s^2) = \tr(S_{A}\Theta(1)), \\
	\text{ and } \sum_{s=1}^m \gl_s a_{s,s}(\hat{q}_s-q_s^2)        & = \tr(S_{A}\Theta(0)).
\end{align*}
We also define the $m\times m$ matrices (not necessarily symmetric)
\begin{align}\label{Bhat}
	\hat{B}(i):=B(i)\gD^2\gL, \qquad \text{ for } i=0,1,2,
\end{align}
Thus, from equations~\eqref{U2},~\eqref{U1} and~\eqref{U0}, for any $m\times m$ symmetric matrix $S$ we have
\begin{align}
	\begin{split}
		\nu( \ovR_{12}^{\T}S \ovR_{12}) & = \nu( \ovR_{12}^{\T}S \hat{B}(2) \ovR_{12}) - 4 \nu(\ovR_{12}^{\T}S \hat{B}(1)\ovR_{13}) \\
		&\qquad + 3\nu(\ovR_{12}^{\T}S \hat{B}(0)\ovR_{34}) + \frac1N\tr(S \Theta(2)) + \fR,\\
		\nu(\ovR_{12}^{\T}S\ovR_{13})& = \nu(\ovR_{12}^{\T}S \hat{B}(1)\ovR_{12}) + \nu(\ovR_{12}^{\T}S (\hat{B}(2)-2\hat{B}(1)-3\hat{B}(0))\ovR_{13})\\
		&\qquad + \nu(\ovR_{12}^{\T}S (6\hat{B}(0)-3\hat{B}(1))\ovR_{34}) +\frac1N\tr(S \Theta(1)) + \fR,\\
		\nu(\ovR_{12}^{\T}S\ovR_{34})&= \nu(\ovR_{12}^{\T}S \hat{B}(0)\ovR_{12}) + \nu(\ovR_{12}^{\T}S (4\hat{B}(1)-8\hat{B}(0))\ovR_{13})\\
		&\qquad + \nu(\ovR_{12}^{\T}S (\hat{B}(2)-8\hat{B}(1)+10\hat{B}(0))\ovR_{34}) +\frac1N\tr(S \Theta(0))+ \fR.
	\end{split}
	\label{eq:1}\end{align}
The system of equations~\eqref{eq:1} can be written as a system of linear equations in the variables $\nu((R_{12}^{(s)}-q_{s})(R^{(t)} _{\ell\ell'}- q_{t})), s\le t, \ell<\ell'$ as follows.

We define the three $m\times m$ matrices
\begin{align*}
	U(0):=\nu(\ovR_{12}\ovR_{34}^{\T}),\ U(1):=\nu(\ovR_{12}\ovR_{13}^{\T}),\ U(2):=\nu(\ovR_{12}\ovR_{12}^{\T}).
\end{align*}
Thus equation~\eqref{eq:1} can be restated in the following way: for any symmetric $m\times m$ matrix $S$
\begin{align}\label{eq:2}
	\tr(SU(i)) = \sum_{j=0}^2\sum_{k=0}^2 C_{ij}(k)\tr(S\hat{B}(k)U(j)) + \frac1N\tr(S \Theta(i)) + \fR,\qquad i=0,1,2,
\end{align}
where
\begin{align*}
	C(0):=
	\begin{pmatrix}
		10 & -8 & 1 \\6 &-3&0\\3 &0&0
	\end{pmatrix}
	, \quad C(1):=
	\begin{pmatrix}
		-8 & 4 & 0 \\-3&-2&1\\0 &-4&0
	\end{pmatrix}
	\text{ and } C(2):=
	\begin{pmatrix}
		1 & 0 & 0 \\0&1&0\\0&0&1
	\end{pmatrix}
\end{align*}
(rows and columns are indexed by $0,1,2$). 

The matrices, $C(0), C(1), C(2)$, commute with each other and thus are simultaneously ``upper-triangularizable''. This fact is crucially used in the variance computation for the overlap in Section~\ref{sec:varoverlap}. Explaining the commutativity property of the $C(\cdot)$ matrices involves understanding the underlying algebraic structure of the Gibbs measure and is an interesting open question. Note that this phenomenon appears in both SK and MSK models. One can guess that this is probably due to some high-level symmetry among replicas in the RS regime. Also, in the MSK model, it is likely to be connected to the synchronization property of overlap discovered by Panchenko (see~\cite{Pan15}).

Note that, the upper bound on $\fR$ depends on $S$ only through $\max_{1\le r,t\le m}|S_{r,t}|$. We denote by $\ve_{i}$, the column vector with $1$ in the $i$-th coordinate and zero elsewhere. Taking $S=(\ve_{l}\ve_{k}^{\T}+\ve_{k}\ve_{l}^{\T})/2$ for all possible choices of $1\le k\le l \le m$, from equation~\eqref{eq:2}, we get
\begin{align}\label{eq:3}
	U(i) = \sum_{j=0}^2\sum_{k=0}^2 C_{ij}(k)\cdot\sym(\hat{B}(k)U(j)) + \frac1N\Theta(i) + \mfR, \quad \text{for} \quad i=0,1,2,
\end{align}
where $\sym(A):=(A+A^{\T})/2$ for a square matrix $A$ and $\mfR$ is a $m\times m$ matrix with
\[ \max_{p,q}|\mfR_{p,q}|\le K\left(N^{-3/2} + \sum_{s=1}^{m}\nu(|R_{12}^{(s-)}-q_s|^{3})\right). \]
\begin{rem}
	The above argument does not require the positive semi-definiteness of $\gD^2$. It seems that at least in a high-temperature region, the RS solution given by Parisi formula~\cite{Pan15} is still valid even for indefinite $\gD^2$.
\end{rem}

\section{Variance of overlap} \label{sec:varoverlap}
In Section~\ref{sec:cavity}, we introduced a species-wise cavity method for the MSK model to derive a linear system involving the quadratic form of overlap. In this section, by studying the linear system, we solve the overlap vectors' variance-covariance structure. Along the way, we obtain the AT-line condition in the MSK model.

\subsection{Proof of Theorem~\ref{thm:lyap}}
We recall the notations
\begin{align*}
	b_{s}(0) & =\gb^2 (\hat{q}_{s}-q_s^2),\quad b_{s}(1) =\gb^2 q_s(1-q_s),\quad b_{s}(2) = \gb^2(1-q_s^2)
\end{align*}
for $s=1,2,\ldots,m$; and
\begin{align*}
	B(i) = \diag(b_1(i), b_2(i),\ldots, b_m(i)),\quad \Theta(i) =\gb^{-2}B(i)\gL^{-1},\quad \hat{B}(i) =B(i)\gD^2\gL
\end{align*}
for $i=0,1,2.$ Then with
\begin{align*}
	C(0):=
	\begin{pmatrix}
		10 & -8 & 1 \\6 &-3&0\\3 &0&0
	\end{pmatrix}
	, \quad C(1):=
	\begin{pmatrix}
		-8 & 4 & 0 \\-3&-2&1\\0 &-4&0
	\end{pmatrix}
	\text{ and } C(2):=
	\begin{pmatrix}
		1 & 0 & 0 \\0&1&0\\0&0&1
	\end{pmatrix}
\end{align*}
(row and column indexed by $0,1,2$), we have from~\eqref{eq:3}
\begin{align*}
	U(i) = \sum_{j=0}^2\sum_{k=0}^2 C_{ij}(k)\cdot\sym(\hat{B}(k)U(j)) + \frac1N\Theta(i) + \mfR,\qquad i=0,1,2,
\end{align*}
where $\sym(A):=(A+A^{\T})/2$.

The matrices, $C(0),C(1),C(2)$, commute with each other and thus are simultaneously ``upper-triangularizable''. In particular, with
\[ V=
	\begin{pmatrix}
		-3 & 2 & 0 \\3 &-4&1\\1&-2&1
	\end{pmatrix}
\]
and defining $T(k):=VC(k)V^{-1}, k=0,1,2$ we have
\begin{align*}
	T(0) =
	\begin{pmatrix}
		3 & -3 & 0 \\
		0 & 3  & 0 \\
		0 & 0  & 1
	\end{pmatrix}
	,
	T(1) =
	\begin{pmatrix}
		-4 & 2  & 0  \\
		0  & -4 & 0  \\
		0  & 0  & -2
	\end{pmatrix}
	,
	T(2)=
	\begin{pmatrix}
		1 & 0 & 0 \\
		0 & 1 & 0 \\
		0 & 0 & 1
	\end{pmatrix}.
\end{align*}
Define $\hat{U}(l):=\sum_{i=0}^{2}V_{li}U(i)$ for $l=0,1,2$, \ie\
\begin{align*}
	\hat{U}(0) & := -3U(0)+2U(1),     \\
	\hat{U}(1) & := 3U(0)-4U(1)+U(2), \\
	\hat{U}(2) & := U(0)-2U(1)+U(2).
\end{align*}
Similarly, we define
\[ \hat{\Theta}(l):=\sum_{i=0}^{2}V_{li}\Theta(i) \text{ for } l=0,1,2.
\]
We have for $l=0,1,2$,
\begin{align*}
	\hat{U}(l)
	 & =\sum_{i=0}^{2}V_{li} \sum_{j=0}^2\sum_{k=0}^2 C_{ij}(k)\cdot\sym(\hat{B}(k)U(j)) + \frac1N \sum_{i=0}^{2}V_{li}\Theta(i) + \mfR \\
	 & =\sum_{j=0}^2 \sum_{k=0}^2(VC(k))_{lj}\cdot\sym(\hat{B}(k)U(j)) + \frac1N \hat{\Theta}(l) + \mfR                                 \\
	 & =\sum_{k=0}^2 \sym\bigl(\hat{B}(k)\cdot \sum_{j=0}^2(T(k)V)_{lj} U(j) \bigr) + \frac1N \hat{\Theta}(l) + \mfR.
\end{align*}
Simplifying, we get
\begin{align}\label{Lyapunov}
	\begin{split}
		\sym\biggl( \bigl(I-\sum_{i=0}^{2}V_{li}\hat{B}(i)\bigr) \cdot \hat{U}(l)\biggr) &= \frac1N \hat\Theta(l) + \mfR,\qquad l=1,2,\\
		\text{ and }
		\sym\biggl( \bigl(I-\sum_{i=0}^{2}V_{1i}\hat{B}(i)\bigr) \cdot \hat{U}(0)\biggr) &= \sym\biggl( \sum_{i=0}^{2}V_{0i}\hat{B}(i) \cdot \hat{U}(1)\biggr) + \frac1N \hat\Theta(0) + \mfR.
	\end{split}
\end{align}

Now consider the $m\times m$ diagonal matrices $\gC, \gC', \gC''$, whose $s$-th diagonal entries are respectively given by
\begin{align*}
	\gc_s   & = \gb^{-2}(b_s(0)-2b_s(1)+b_s(2)) = 1-2q_{s}+\hat{q}_s,    \\
	\gc'_s  & = \gb^{-2}(3 b_s(0)-4b_s(1)+b_s(2))= 1 -4q_{s}+3\hat{q}_s, \\
	\gc''_s & = \gb^{-2}(-3b_s(0)+2b_s(1)) =2q_s +q_s^2 -3\hat{q}_s.
\end{align*}
Then,
\begin{align*}
	\sum_{i=0}^{2}V_{2i}\hat{B}(i) & = \gb^{2}\gC\gD^{2}\gL,\quad
	\sum_{i=0}^{2}V_{1i}\hat{B}(i) = \gb^{2}\gC'\gD^{2}\gL,\quad
	\sum_{i=0}^{2}V_{0i}\hat{B}(i) = \gb^{2}\gC''\gD^{2}\gL,                                                               \\
	\hat\Theta(2)                  & = \gC\gL^{-1},\quad \hat\Theta(1) = \gC'\gL^{-1},\quad \hat\Theta(0) = \gC''\gL^{-1}.
\end{align*}
Combining with~\eqref{Lyapunov} we have the result.\hfill\qed

Before going to the proof of Theorem~\ref{thm:varoverlap}, we will prove the following lemma, which essentially solves the continuous Lyapunov equation. Recall that a square matrix $A$ is \emph{stable} if all the eigenvalues of $A$ have a strictly negative real part.

\begin{lem}\label{lem:cle}
	Let $A$ be a $m\times m$ stable matrix, and $C$ be a symmetric matrix. Suppose that the symmetric matrix $X$ satisfies the equation,
	\[
		\sym(AX) = - C.
	\]
	Then we have
	\[
		X=\int_{0}^{\infty}e^{\frac{t}{2}A}Ce^{\frac{t}{2}A^{\T}}dt.
	\]
\end{lem}
\begin{proof}[Proof of Lemma~\ref{lem:cle}]
	First we consider the case when $A$ is stable and similar to a diagonal matrix, \ie\ $A=SDS^{-1}$ for a diagonal matrix $D$ with negative diagonal entries and an invertible matrix $S$. We can write $\sym(AX)= S\sym(DS^{-1} XS^{-\T})S^{\T}$ and solve for $\sym(DY) = - S^{-1}CS^{-\T}$ where $Y=S^{-1} XS^{-\T}$. Furthermore,
	\[
		\int_{0}^{\infty} Se^{\frac{t}{2}D}S^{-1}CS^{-\T}e^{\frac{t}{2}D}S^{\T}dt = \int_{0}^{\infty}e^{\frac{t}{2}A}Ce^{\frac{t}{2}A^{\T}}dt.
	\]
	Thus, w.l.o.g.~$A$ can be taken as a diagonal matrix.

	We write $\vect(A)$ as the $m^{2}\times 1$ vector formed by stacking the columns of $A$, \ie\ $\vect(A) = \sum_{s=1}^{m}e_s\otimes Ae_s=\sum_{s=1}^{m}A^{\T}e_s\otimes e_s$, where $e_s$ is the canonical basis vector with 1 at $s$-th entry and 0 elsewhere, $\otimes$ is the Kronecker product of matrices. We define a linear map $\phi$ from $m \times m$ matrices to $m^2 \times m^2$ matrices:
	\[ \phi(A) := \frac 12(A\otimes I + I\otimes A). \]
	One can easily check that, $\vect(AXB^{\T})=(B\otimes A)\vect(X)$ for any symmetric matrix $X$, and thus we have
	\[ \vect(\sym(AX)) = \phi(A)\vect(X). \]
	From $\sym(AX)=-C$, we get
	\begin{align}\label{eq:3a}
		\phi(A)\vect(X) = -\vect(C).
	\end{align}
	Since $A$ is negative definite diagonal matrix, $\phi(A)$ is also a negative definite diagonal matrix and hence invertible with inverse given by
	\[ (\phi(A))^{-1}=-\int_{0}^{\infty}e^{t\phi(A)}dt=-\int_{0}^{\infty}e^{\frac{t}{2}A}\otimes e^{\frac{t}{2}A}dt. \]
	Thus, we have
	\[
		\vect(X)=-(\phi(A))^{-1}\vect(C) = \int_{0}^{\infty} (e^{\frac{t}{2}A}\otimes e^{\frac{t}{2}A}) \vect(C) dt = \int_{0}^{\infty} \vect(e^{\frac{t}{2}A} Ce^{\frac{t}{2}A} ) dt
	\]
	and
	$$
		X= \int_{0}^{\infty}e^{\frac{t}{2}A} Ce^{\frac{t}{2}A}dt.
	$$
	In the general case, by Jordan decomposition, $A$ is similar to an upper triangular matrix with diagonal entries having negative real parts. The exact proof goes through.
\end{proof}

\subsection{Proof of Theorem~\ref{thm:varoverlap}}
First we note that $0\le \gc_{s}\le 1, -1< \gc_{s}'\le \gc_{s}\le 1$ and thus $\rho(\gC\gD^2\gL)\le \rho(\gD^2\gL), \rho(\gC'\gD^2\gL)\le \rho(\gD^2\gL)$.
If $\gb^{2} <\gb_c^2=\rho(\gD^2\gL)^{-1}$, it is easy to see that the matrices $\gb^2\gC\gD^{2}\gL-I, \gb^2\gC'\gD^{2}\gL-I$ are stable. Furthermore, we have
\begin{align*}
	\hat{U}(2) & = \frac1N\cdot \int_{0}^{\infty} \exp\left(-\frac{t}{2}(I-\gb^2\gC\gD^{2}\gL)\right)\gC\gL^{-1}\exp\left(-\frac{t}{2}(I-\gb^2\gL\gD^{2}\gC)\right)dt + \mfR                \\
	           & = \frac1N\cdot \int_{0}^{\infty} \gC^{\half}\gL^{-\half} \exp\left(-t(I-\gb^2\gC^{\half}\gL^{\half}\gD^{2}\gL^{\half}\gC^{\half})\right)  \gC^{\half}\gL^{-\half}dt + \mfR \\
	           & = \frac1N\cdot \gC^{\half}\gL^{-\half} (I-\gb^2\gC^{\half}\gL^{\half}\gD^{2}\gL^{\half}\gC^{\half})^{-1} \gC^{\half}\gL^{-\half}+ \mfR                                     \\
	           & = \frac1N\cdot \gC (I-\gb^2\gD^{2}\gL\gC)^{-1} \gL^{-1}+ \mfR.
\end{align*}
In the second equality, we used the fact that $ABA^T = B^{1/2}(B^{-1/2}AB^{1/2})(B^{-1/2}AB^{1/2})^T B^{1/2}$ with $A=\exp\left(-\frac{t}{2}(I-\gb^2\gC\gD^{2}\gL)\right)$ and $B:=\Gamma \gL^{-1}$, a diagonal matrix with positive entries.

Similarly, we get
\begin{align*}
	\hat{U}(1)
	 & = \frac1N\cdot \gC' (I-\gb^2\gD^{2}\gL\gC')^{-1} \gL^{-1}+ \mfR
\end{align*}
and
\begin{align}\label{eq:var-overlap-3}
	\hat{U}(0)
	 & =  \int_{0}^{\infty} \exp\left(-\frac{t}{2}(I-\gb^2\gC'\gD^{2}\gL)\right) \sym\bigl( \gb^{2}\gC''\gD^{2}\gL\cdot \hat{U}(1)\bigr) \exp\left(-\frac{t}{2}(I-\gb^2\gL\gD^{2}\gC')\right)dt \\
	 & +\frac1N\cdot \int_{0}^{\infty} \exp\left(-\frac{t}{2}(I-\gb^2\gC'\gD^{2}\gL)\right) \gC''\gL^{-1} \exp\left(-\frac{t}{2}(I-\gb^2\gL\gD^{2}\gC')\right)dt  +\mfR.
\end{align}
Finally, we use the fact that $N\mfR$ converges to $0$ entrywise as $N\to\infty$.\hfill\qed

\section{Uniqueness of $\vq_{\star}$ and Replica Symmetry Breaking }\label{sec:RSB}
In this section, for \pd~$\gD^2$, we prove the AT line condition of the MSK model when $h>0$, beyond which the MSK model is in replica symmetry breaking phase. From classical literature of the SK model, we know that the uniqueness of $\vq_{\star}$ is essential to characterize the AT line condition. In the SK model, the uniqueness was proved using contraction argument for $\gb$ small, and the Latala-Guerra Lemma~\cite{Tal11a}*{Proposition 1.3.8} gives the uniqueness result for $h>0$. In the MSK model, we prove the uniqueness of $\vq_{\star}$ for small $\beta$ by extending the contraction argument, but extending Latala-Guerra Lemma becomes more challenging, which is about analysis of a complicated nonlinear system~\eqref{syseq}. In~\cite{BSS19}, they use an elementary approach to prove this for \pd~$2$-species model, but the idea seems difficult to be generalized. We first prove the uniqueness result for small $\gb$ in Theorem~\ref{thm0}, then we prove Proposition~\ref{prop:indf-uniq} which gives the uniqueness result for \indf~ 2-species models when $h>0$.

\subsection{Proof of Theorem~\ref{thm0}}
We need to prove that the following system
\begin{align}\label{sys1}
	q_s = \E \tanh^2(\gb \eta \sqrt{(\gD^2\gL \vq)_{s}}+h),\qquad s=1,2,\ldots,m
\end{align}
has a unique solution where $\vq=(q_1,q_2,\ldots,q_m)^\T$.
We rewrite the equations in terms of $\vx = (x_1,x_2,\ldots,x_m)^\T:=\gL^{\half} \vq$. Let $f(y) := \tanh^2(y)$. Then the system of equations~\eqref{sys1} is equivalent to
\begin{align}
	x_s = {\psi}_s,\qquad s=1,2,\ldots,m
\end{align}
where
$$
	\psi_s
	:= {\gl^{\half}_s}\E f(\gb \eta \sqrt{(\gD^2\gL^{\half}\vx)_s} + h).
$$
We define
${\Psi}(\vx) = (\psi_1,\psi_{2}, \ldots, \psi_m )$.
It's easy to compute the Jacobian matrix of the map $\Psi$,
\[
	J({\Psi})= \left(\!\!\!\left(\frac{\partial \psi_s}{\partial x_t} \right)\!\!\!\right)_{s,t=1}^{m}
\]
where
$$
	\frac{\partial \psi_s}{\partial x_t}
	=\frac{\gb {\gl^{\half}_s}\gD_{s,t}^2{\gl^{\half}_t}}{2\sqrt{(\gD^2\gL^{\half}\vx)_s}} \E \eta f^{\prime}(\gb \eta \sqrt{(\gD^2\gL^{\half}\vx)_s}+h)
	= \frac{\gb^2}{2} {\gl^{\half}_s} \gD_{st}^2 {\gl^{\half}_t} \E f^{\prime \prime}(\gb \eta \sqrt{(\gD^2\gL^{\half}\vx)_s}+h)
$$
and the last equality follows by Gaussian integration by parts. Now, note that
\[
	f^{\prime}(y) = 2\cdot \frac{\tanh (y)}{\cosh^2 (y)}, \ f^{\prime \prime}(y) = 2\cdot \frac{1- 2 \sinh^2(y)}{\cosh^4(y)} \in [-2,2].
\]
In particular, we have
\begin{align*}
	J({\Psi}) = \gb^2 A \gL^{\half} \gD^2 \gL^{\half}
\end{align*}
where $A=\diag(a_{1},a_{2},\ldots,a_{m})$ and $
	a_s := \frac12\E f^{\prime \prime}(\gb \eta  \sqrt{(\gD^2\gL^{\half}\vx)_s}+h)\in [-1,1]
$
for $s=1,2,\ldots,m$. Thus
\[
	\norm{J(\Psi)} \le \gb^2 \norm{\gL^{\half}\gD^2 \gL^{\half}}= \gb^2 \rho(\gL^{\half}\gD^2 \gL^{\half})=\gb^2 \rho(\gD^2 \gL).
\]
In particular, $\gb^2 < \frac{1}{\rho(\gD^2 \gL)}$ implies that $\norm{J({\Psi})} <1$, \ie\ ${\Psi}$ is a contraction and the system~\eqref{sys1} has a unique solution. \hfill\qed

Next, we present an elementary approach to prove Proposition~\ref{prop:indf-uniq}, which concerns the uniqueness of $\vq_{\star}$ for \indf~$\gD^2$ with $m=2$.
\subsection{Proof of Proposition~\ref{prop:indf-uniq}}

We use an elementary approach to prove Proposition~\ref{prop:indf-uniq}. Recall $\vQ = A \vq$, where $A=\gD^2\gL$. For \indf~$\gD^2$, we have $\det(A) <0$, then
\[A^{-1} =
	\begin{pmatrix}
		-a & b  \\
		c  & -d
	\end{pmatrix} \quad \text{where} \quad a,b,c,d >0.
\]
Then we can rewrite the equation,
\[
	\begin{cases}
		q_1 = \E \tanh^2(\gb \eta \sqrt{Q^1}+h), \\
		q_2 = \E \tanh^2(\gb \eta \sqrt{Q^2}+h).
	\end{cases}
\]
as
\begin{align}\label{eq:trans}
	\begin{cases}
		-aQ^1+bQ^2 = \E \tanh^2(\gb \eta \sqrt{Q^1}+h), \\
		cQ^1-dQ^2 = \E \tanh^2(\gb \eta \sqrt{Q^2}+h).
	\end{cases}
\end{align}
Let $g_1(x) := \frac{1}{b}(\E \tanh^2(\gb \eta \sqrt{x}+h)/x+a)$ and $g_2(x) := \frac{1}{c}(\E \tanh^2(\gb \eta \sqrt{x}+h)/x+d)$, by the classical Latala-Guerra lemma~\cite{Tal11a}*{Appendix A.14}, $g_1(x),g_2(x)$ are both strictly decreasing on $\mathbb{R}^+$. The equation~\eqref{eq:trans} can be rewritten further:
\begin{align}\label{eq:trans2}
	\begin{cases}
		Q^1 = Q^2g_2(Q^2), \\
		Q^2 = Q^1g_1(Q^1).
	\end{cases}
\end{align}
From this, we get
\[
	g_1(Q^1) g_2(Q^2) =1,
\]
then $Q^2=g_2^{-1}\left(\frac{1}{g_1(Q^1)}\right) $ is a decreasing function of $Q^1$. On the other hand, taking cross product of equation~\eqref{eq:trans2}, we have

\begin{align}\label{eq:trans3}
	(Q^2)^2 g_2(Q^2) =( Q^1)^2g_1(Q^1).
\end{align}

Let $h_i(x) = x^2 g_i(x), i=1,2.$ Next we show $h_1,h_2$ are both increasing functions on $\mathbb{R}^+$. Since
\begin{align*}
	h_1(x) = x^2g_1(x) & = \frac{x\E \tanh^2(\gb \eta \sqrt{x}+h)+ax^2}{b}                                                                                    \\
	                   & = \int_{-\infty}^{\infty} \frac{\sqrt{x}}{b} \tanh^2(y) \frac{1}{\sqrt{2\pi \gb^2}} e^{-\frac{(y-h)^2}{2\gb^2x}} dy+ \frac{ax^2}{b},
\end{align*}
where we expand the expectation part as a Gaussian integral. For this integral form, it's easy to check $h_1(x)$ is an strictly increasing function of $x \in \mathbb{R}^+$. Similarly, we can prove $h_2$ is also strictly increasing. From~\eqref{eq:trans3}, we have
\[
	h_1(Q^1) = h_2(Q^2).
\]
Thus $Q_2 = h_2^{-1}\left(h_1(Q^1) \right)$ is strictly increasing as a function of $Q^1$. Recall $Q_2 = g_2^{-1}\left(\frac{1}{g_1(Q^1)}\right)$ is a decreasing function of $Q^1$. Therefore, we conclude the uniqueness of $Q^1$, similarly for $Q^2$. The proof is complete. \hfill \qed


In the following part, we will prove Theorem~\ref{AT line}. As we mentioned before, we assume $\gD^2$ is \pd~because the Parisi formula is only known in that case. We further assume the uniqueness of $\vq_{\star}$ when $h>0$, in order for accurate and correct characterization of AT line condition.

Consider $\vp \in [0,1]^m$ with $p^s\ge q_{\star}^s$ for all $s=1,2,\ldots, m$, define
\begin{align}\label{RSB-V}
	V(\vp) = \frac{\partial \sP_{\text{1RSB}}(\vq_{\star},\vp,\zeta)}{\partial \zeta} \bigg |_{\zeta=1} ,
\end{align}
where $\sP_{\text{1RSB}}$ is Parisi functional of 1-step replica symmetry breaking. The explicit expression and related computations can be found in the appendix of~\cite{BSS19}. The following Lemma gives some useful properties of $V(\vp)$.
\begin{lem}
	[{\cite{BSS19}*{Lemma~3.1}}]
	Fix any $\vq=\vq_{\star} \in \cC(\gb,h)$, The following identities hold:
	\begin{enumerate}
		\item $V(\vq_{\star}) =0$.
		\item $\nabla V(\vq_{\star}) =0$.
		\item $HV(\vq_{\star}) = \gb^2\gL(\gb^2\gD^2\gL\Gamma\gD^2-\gD^2)\gL$, where $\gL, \Gamma$ are diagonal matrices defined in Theorem~\ref{AT line}.
	\end{enumerate}
\end{lem}
The following Corollary characterizes the RSB condition using $HV(\vq_{\star})$ in the above lemma.
\begin{cor}
	[{\cite{BSS19}*{Corollary 3.2}}]\label{RSB-cor}
	Assume the Parisi formula~\eqref{par}, and that $\text{RS}(\gb,h) = \sP_{\RS}(\vq_{\star})$ for some $\vq_{\star} \in \cC(\gb,h)$. If there exists a $\vx \in \bR^m$ with all nonnegative entries such that $\vx^{\T}HV(\vq_{\star})\vx>0$, then
	\[
		\lim_{N \to \infty}F_N < \RS(\gb,h) .
	\]
\end{cor}

Now we use the characterization of RSB in Corollary~\ref{RSB-cor} to prove Theorem~\ref{AT line}.
\subsection{Proof of Theorem~\ref{AT line}}
Based on the definition, the matrix $\Gamma^{\half}\gL^{\half}\gD^2\gL^{\half}\Gamma^{\half}$ has all entries positive, by Perron-Frobenius theorem, we know that there exists an eigenvector $\vu=(u_1,u_2,\ldots, u_m) \in \bR^m $ associated to the largest eigenvalue $\rho(\gC\gD^2\gL)>0$ such that $u_s >0 $ for $s=1,2,\ldots,m$. Without loss of generality, we assume $\vu$ is a unit vector, \ie\ $\vu^{\T}\vu=1$. Take $\vx=\gL^{-\half}\Gamma^{\half}\vu$ which has positive entries. Then
\begin{align*}
	\vx^{\T}\left(\gL(\gb^2\gD^2\gL\Gamma\gD^2-\gD^2)\gL\right)\vx & = \vu^{\T}\left[\Gamma^{\half}\gL^{-\half}\gL(\gb^2 \gD^2\gL\Gamma\gD^2-\gD^2)\gL\gL^{-\half}\Gamma^{\half}\right]\vu             \\
	                                                               & = \vu^{\T}\left(\gb^2 [(\gL\Gamma)^{\half}\gD^2(\gL\Gamma)^{\half}]^2\vu-[(\gL\Gamma)^{\half}\gD^2(\gL\Gamma)^{\half}]\vu \right) \\
	                                                               & = \rho(\gC\gD^2\gL)\left({\gb^2\rho(\gC\gD^2 \gL)-1}\right).
\end{align*}
where the last step uses the fact that $\vu$ is an eigenvector associated with $\rho(\gC\gD^2\gL)$ and $\vu^{\T}\vu =1$. If $\gb^2 > \frac{1}{\rho(\gC\gD^2\gL)}$, we have
\[ \vx^{\T}\left(\gL(\gb^2\gD^2\gL\Gamma\gD^2-\gD^2)\gL\right)\vx >0 ,\]
\ie\ $\vx^{\T}HV(\vq_{\star})\vx > 0 $. By Corollary~\ref{RSB-cor}, we have
\[ \lim_{N \to \infty} F_N < \RS(\gb,h) ,\]
and this completes the proof. \hfill\qed

\section{Discussions and Further Questions}\label{sec:openqn}
In this section, we discuss some open questions spread out in the previous sections.

\begin{itemize}
	\item In Theorem~\ref{RS solution}, the RS solution in \indf~$\gD^2$ case, obtained using Guerra's interpolation has the same expression as evaluating the Parisi functional $\sP_{\text{RS}}$ at $\vq_{\star}$, while $\vq_{\star}$ is not a minimizer of $\sP_{\text{RS}}$ any more as in \pd~case. This fact suggests that a modified Parisi formula as conjectured in~\cite{BGG11} should be true at least in the RS regime. Our argument verifies this in part of the RS regime. The question is whether one can prove it in the whole RS regime while adapting Guerra's interpolation?
	\item For \pd~$\gD^2$, when $h>0$, analyze the uniqueness of solution to the nonlinear system:
	      \begin{align*}
		      q_s = \E \tanh^2(\gb \eta \sqrt{(\gD^2\gL\vq)_s+h}, \quad \text{for} \quad s=1,2,\ldots, m.
	      \end{align*}
	      As we mentioned, the uniqueness condition is essential to characterize AT line in Theorem~\ref{AT line}. In~\cite{BSS19}, it was proved in the 2-species case, where one needs to identify the signs of matrix entries, but this is nearly impossible for general $m>2$. Similarly for \indf~case with $m>2$.
	\item In Theorem~\ref{AT line}, we proved the RSB condition when $h>0$. In the case $h=0$, to prove RSB when $\gb>\gb_c$ is still challenging. In that case, $\cC(\gb,0)=\{0,\vq_{\diamond}\}$, $\RS(\gb,0):=\sP_{\RS}(\vq_{\diamond})$ achieves the minimum at $\vq_{\diamond}$, to prove $\lim_{N \to \infty} F_N(\gb)< \RS(\gb,0)$ is still not clear and depends on a good understanding of $\vq_{\diamond}$.

	\item Can one prove or disprove the Conjecture~\ref{AT conj}? In this case, one also needs to understand the solution to~\eqref{syseq} for \indf~$\gD^2$.

	\item In Theorem~\ref{thm:varoverlap}, we proved the asymptotic variance of overlap, then the next question will be a central limit theorem. In the classical SK model, Talagrand~\cite{Tal11a} proved it using the moment method, but adapting the idea for the MSK model seems hopeless because now we have overlap vectors and many computations involve matrices. One has to find a different method to prove it.

	\item Note that, the function
	      \[ \vX(t):=\int_{0}^{-\ln(1-t)} e^{-\frac{s}{2}(I-\gb^2\gC\gD^2\gL)} \gC e^{-\frac{s}{2}(I-\gb^2\gL\gD^2\gC)}ds,\ t\in[0,1] \]
	      is the unique solution of
	      \begin{align*}
		      (1-t)\vX''(t) = \gb^2\cdot \sym(\vX'(0)\gD^2\gL\vX'(t)) \text{ with } \vX(0)=0, \vX'(0) = \gC .
	      \end{align*}
	      when $I-\gb^2\gC\gD^2\gL$ has all eigenvalues strictly positive. Is it possible to connect this equation with Guerra's interpolation?

	\item Finally, continuous Lyapunov equation arises in studying the variance-covariance matrices for OU-type SDE of the form $dX_{t}=AX_{t}dt+CdB_{t}$, where $A$ is a  stable matrix and  $\lim_{t\to\infty}\var(X_{t})=\int_{0}^{\infty}e^{tA}CC^{\T}e^{tA^{\T}}dt$. It will be interesting to connect the appearance of continuous Lyapunov equation in Replica Symmetric MSK model with an appropriate SDE.

	      %
	      %
	      %
\end{itemize}

\noindent{\bf Acknowledgments.}
The authors would like to thank Erik Bates, Jean-Christophe Mourrat, Dimitry Panchenko for reading the manuscript and insightful comments, and the two anonymous referees for providing helpful suggestions and additional references that improved the clarity and presentation of the paper.

\bibliography{RSB.bib}

\end{document}